\documentclass{math-note}

\usepackage[english]{babel}
\usepackage{mathtools,amssymb,amsthm,mathrsfs}
\usepackage{booktabs,array}
\usepackage[hyphens]{url}
\usepackage[hidelinks,
  pdftitle={A Two-Stage Construction of Positive Curvature on the Gromoll--Meyer Sphere},
  pdfauthor={Shengtao Guo, Ethan X. Fang, Junwei Lu}]{hyperref}
\font\gmtitlefont=cmr17 at 14.4pt 
\title{\gmtitlefont A Two-Stage Construction of Positive Curvature on the Gromoll--Meyer Sphere}
\date{}
\author{
  Shengtao Guo \qquad
  Ethan X. Fang \qquad
  Junwei Lu\thanks{Department of Biostatistics, Harvard T.H. Chan School of
  Public Health. Email: \texttt{junweilu@hsph.harvard.edu}.}
}
\numberwithin{equation}{section}
\newtheorem{theorem}{Theorem}[section]
\newtheorem{proposition}[theorem]{Proposition}
\newtheorem{lemma}[theorem]{Lemma}
\newtheorem{corollary}[theorem]{Corollary}
\theoremstyle{definition}

\theoremstyle{remark}

\newcommand{\R}{\mathbb R}
\newcommand{\HH}{\mathbb H}
\newcommand{\Sp}{\operatorname{Sp}}
\newcommand{\SO}{\operatorname{SO}}
\newcommand{\Ad}{\operatorname{Ad}}
\newcommand{\ad}{\operatorname{ad}}
\newcommand{\im}{\operatorname{Im}}
\newcommand{\re}{\operatorname{Re}}
\newcommand{\tr}{\operatorname{tr}}
\newcommand{\spanof}{\operatorname{span}}
\newcommand{\diag}{\operatorname{diag}}
\newcommand{\Gr}{\operatorname{Gr}}
\newcommand{\dist}{\operatorname{dist}}

\newcommand{\Afamily}{\mathscr A}
\newcommand{\Bfamily}{\mathscr B}
\newcommand{\ZA}{\mathcal Z_{\Afamily}}
\newcommand{\ZB}{\mathcal Z_{\Bfamily}}
\newcommand{\calQ}{\mathscr Q}
\newcommand{\ip}[2]{\langle #1,#2\rangle}
\newcommand{\norm}[1]{\lVert #1\rVert}

\begin{document}
\maketitle

\begin{abstract}
We construct an explicit one-parameter family of smooth metrics on the
Gromoll--Meyer exotic seven-sphere,
converging in \(C^\infty\) to a fixed further Cheeger deformation of the
Eschenburg--Kerin metric and having strictly positive sectional curvature
for all sufficiently small positive parameter values. The first perturbation
preserves the totally geodesic flats of one zero-plane family while making
curvature positive near the other; the second removes the remaining zero
curvature.
The metric and the proof were discovered by the Odin Automatic AI \mbox{Research Agent}.
\end{abstract}

\section{Introduction}

Gromoll and Meyer realized an exotic seven-sphere \(\Sigma\) as a free
biquotient of \(\Sp(2)\) and equipped it with nonnegative sectional
curvature~\cite[Theorems 1 and 2]{GM74}.

Wilhelm constructed a metric on this sphere with positive sectional
curvature at every plane over an open dense set of
points~\cite[Theorem A]{W01}. Eschenburg and Kerin subsequently obtained
almost positive curvature from two Cheeger deformations and described the
exceptional planes explicitly~\cite[Sections 2--4]{EK08}. Nonnegative
curvature is available on all Milnor spheres by Grove and
Ziller~\cite[Theorems A and B]{GZ00}, and on every exotic seven-sphere by
Goette, Kerin and Shankar~\cite[Theorem A]{GKS20}.

Recent constructions on products of spheres combine Cheeger deformations
with further metric perturbations. Brendle and
Hung~\cite[Introduction]{BH26} obtain positive sectional curvature on
\(S^2\times S^2\) through a third-order perturbation of a Cheeger--M\"uter
metric. Guo, Fang and Lu~\cite[Introduction]{GFL26} obtain positive sectional
curvature on \(S^2\times S^3\) by perturbing a connection metric on a circle
bundle over \(S^2\times S^2\); the base carries a Cheeger metric, and
fibrewise rotation of the perturbation produces a positive second-order
curvature term.

Berman, Cederwall and
Schettini Gherardini computed curvature tensors for explicit
Kaluza--Klein metrics on \(\Sigma\)~\cite[Sections 3 and 5]{BCG25}.
Their positivity estimates concern Ricci curvature.

Petersen and Wilhelm proposed a positive-curvature construction on the
Gromoll--Meyer sphere~\cite[Introduction]{PW08}. Their subsequent study of
deformations preserving totally geodesic flats~\cite[Section~2]{PW10}
provides background for our first perturbation.

Ouyang~\cite[Sections~4--9]{Ouy26} recently presented a construction of
positive sectional curvature on \(\Sigma\) using a fixed operator to define
an inverse-linear deformation of the bi-invariant metric on \(\Sp(2)\).
Its positivity argument uses a cubic coefficient on the remaining critical
set. Our construction and proof were obtained independently of that work.

We start from a fixed further Cheeger deformation \(g_{0,0}\) of the
Eschenburg--Kerin metric and use two successive position-dependent tensor
perturbations. The reference metric and the perturbations are defined in
Section~\ref{sec:metric}.

\begin{samepage}
\begin{theorem}\label{thm:main}
The reference metric \(g_{0,0}\) is a \(C^\infty\) limit of metrics with
strictly positive sectional curvature on the Gromoll--Meyer sphere.
More precisely, the explicit family \(g_\varepsilon\) defined in
\eqref{eq:final-metric} converges to \(g_{0,0}\) in \(C^\infty\) as
\(\varepsilon\to0^+\), and there are constants \(c,\varepsilon_0>0\) such that
\[
 \min_{\sigma\in\Gr_2(T\Sigma)}\sec_{g_\varepsilon}(\sigma)
 \ge c\varepsilon^6
 \qquad(0<\varepsilon<\varepsilon_0).
\]
\end{theorem}
\end{samepage}

The zero set of \(g_{0,0}\) has two disjoint compact submanifolds \(\ZA\) and
\(\ZB\) of zero planes in the Grassmann bundle. We call them the \(\Afamily\) and \(\Bfamily\)
families, and call the corresponding totally geodesic flats the \(\Afamily\) and \(\Bfamily\)
flats. The curvature Hessian
is positive in the normal directions to each. Although their projections
to \(\Sigma\) have singularities, the zero-plane families admit smooth
tubular neighborhoods in which to estimate the perturbed curvature.

On the product source, the first perturbation is \(\varepsilon D^*SD\), where
\(S\) is symmetric and \(D\) vanishes on the tangent spaces of the \(\Afamily\)
flats. It preserves their metric and total geodesy and makes curvature
positive near \(\ZB\). We then perturb by \(tH+t^2J\), with \(H,J\) as in
\eqref{eq:H}--\eqref{eq:J}. The first curvature variation of \(H\) vanishes on \(\ZA\),
while \(J\) makes the reduced second variation positive. The first perturbation
gives an \(\varepsilon^2\) lower bound near \(\ZB\), whereas the second changes
curvature there by \(O(t)\). Taking \(t=\varepsilon^3\) preserves that
bound and gives an \(\varepsilon^6\) lower bound near \(\ZA\).
The same estimates give a region of positive curvature for the two-parameter
family.

\begin{corollary}\label{cor:parameter-region}
There are constants \(c,\eta,\varepsilon_0>0\) such that the metrics
\(g_{\varepsilon,t}\) defined in \eqref{eq:two-parameter-metric} satisfy
\[
 \min_{\sigma\in\Gr_2(T\Sigma)}\sec_{g_{\varepsilon,t}}(\sigma)
 \ge c\min\{t^2,\varepsilon^2\}>0
\]
whenever \(0<\varepsilon<\varepsilon_0\) and
\(0<t\le\eta\varepsilon^2\).
\end{corollary}

The reduced second variation is the minimum over first-order displacements
of the point and the plane. Completing the curvature squares on the product
source gives a lower bound from an explicit feasible solution of linear dual
equations. On \(\ZA\),
we choose the dual variables using a polynomial of a four-by-four matrix.
The remaining estimate reduces to five inequalities in two angular
parameters. The final sign checks use exact rational arithmetic in the
Bernstein basis.

\noindent\textbf{The role of AI in this proof.}
Odin Automatic AI Research Agent was used to construct the metric and the~proof.

\noindent\textbf{Paper organization.}
Section~\ref{sec:metric} defines the metric family.
Section~\ref{sec:zero-families} describes the zero-curvature planes, and
Section~\ref{sec:completion} reduces the main theorem to the \(\Bfamily\) and \(\Afamily\)
variation estimates. Section~\ref{sec:source-variation} derives the
source curvature-variation formula used to prove these estimates in
Sections~\ref{sec:B-bound} and~\ref{sec:A-source}.
Section~\ref{sec:verification} gives the exact computations.
Appendices~\ref{app:clean}--\ref{app:mass-encoding} supply the bracket
estimates, explicit dual vectors, and intrinsic variation calculations.

\section{The metric and the two perturbations}\label{sec:metric}

Let
\(\Sigma=\Sp(2)//\Sp(1)\), with free action
\begin{equation}\label{eq:GM-action}
 q\cdot g=\diag(q,1)\,g\,\diag(q^{-1},q^{-1}).
\end{equation}
This is the inverse-matrix convention of the Gromoll--Meyer
construction~\cite[Section 1]{GM74}. Put
\(Q(X,Y)=-\re\tr(XY)\) on \(\mathfrak{sp}(2)\). Our first coordinate
frame is \(Q\)-orthonormal:
\begin{equation}\label{eq:orth-frame}
 X=\begin{pmatrix}
 d_1&(\rho+\omega)/\sqrt2\\
 -(\rho-\omega)/\sqrt2&d_2
 \end{pmatrix},
 \qquad d_1,d_2,\omega\in\im\HH,\quad \rho\in\R.
\end{equation}
All matrices in this section use the block order
\((d_1,d_2,\rho,\omega)\), with dimensions \((3,3,1,3)\).

The initial left-invariant metric on \(\Sp(2)\) has operator
\begin{equation}\label{eq:P0}
 P_0=\begin{pmatrix}
 3I/8&-I/8&0&0\\
 -I/8&3I/8&0&0\\
 0&0&1&0\\
 0&0&0&I
 \end{pmatrix}.
\end{equation}
It is obtained from \(Q\) by contraction along
\(K=\Sp(1)\times\Sp(1)\) and then along
\(H_\Delta=\Delta\Sp(1)\), as in~\cite[Section 2]{EK08}.
Here contraction means Cheeger's product-submersion
construction~\cite{C73}: an isometric action on \((M,g)\) and a
bi-invariant group metric \(Q_K\) give a new metric on \(M\) by the
quotient of \((M\times K,g+T^{-1}Q_K)\), with \(T>0\).
The projection \((p,k)\mapsto k^{-1}\cdot p\) contracts the metric
in the orbit directions and fixes its orthogonal complement.
The eigenvalues of \(P_0\) on the off-diagonal, diagonal
anti-diagonal and diagonal-diagonal subspaces are respectively
\(1,1/2,1/4\). Thus the two contraction factors in the
Eschenburg--Kerin convention are both \(1/2\).
We denote its quotient under \eqref{eq:GM-action} by \(g_{\rm EK}\).

Write the first row of \(g\in\Sp(2)\) as \((a,b)\), and put
\[
 m=|a|^2,\quad a_0=\re a,\quad b_0=\re b,
 \quad \mathbf a=\im a,\quad\mathbf b=\im b.
\]
For an imaginary quaternion \(v\), \([v]_\times\) denotes the map
\(w\mapsto v\times w\) on \(\im\HH\cong\R^3\). Let \(O\) be the
\(Q\)-orthogonal projection onto the off-diagonal subspace, and define
\begin{equation}\label{eq:annihilator}
 \zeta=\begin{pmatrix}0&\bar a b\\-\bar b a&0\end{pmatrix},
 \qquad
 N=2m(1-m)O-\zeta\zeta^T+
                   \ad_\zeta O\ad_\zeta^T.
\end{equation}
In \(\zeta\zeta^T\), the matrix \(\zeta\) is identified with its
coordinate vector in \eqref{eq:orth-frame}. The map \(NP_0\) will
annihilate the tangent spaces of the \(\Afamily\) flats. Placing it on both sides
of a symmetric operator preserves their induced metric and total geodesy.
The factor \(m-1/2\) below makes the perturbation vanish on \(\Bfamily\), while
its transverse derivatives can change the second curvature variation.
Set
\begin{equation}\label{eq:HN}
 M=\begin{pmatrix}
 -I&2I&0&0\\2I&-I&0&0\\0&0&0&0\\0&0&0&12I
 \end{pmatrix},
 \qquad
 H_N=(m-\tfrac12)P_0NMNP_0.
\end{equation}
The next perturbation pairs a tensor \(H\) whose first curvature variation
vanishes on \(\Afamily\) with a tensor \(J\) that makes the reduced second variation
positive. The following coefficients give one explicit choice; their
curvature properties are proved in Sections~\ref{sec:B-bound}
and~\ref{sec:A-source}. The nonzero blocks are
\begin{align}
 H_{d_2d_2}&=5\sqrt2\,a_0I,&
 H_{d_1\omega}&=-4b_0I-\tfrac{13}{15}[\mathbf b]_\times,
 \label{eq:H}\\
 H_{d_2\omega}&=\tfrac{107}{20}b_0I-4[\mathbf b]_\times,&
 H_{d_2\rho}&=-4\mathbf b,\nonumber\\[2pt]
 J_{d_1d_1}&=(96m+3m^2)I,&
 J_{d_2\omega}&=11b_0[\mathbf a]_\times,
 \label{eq:J}\\
 J_{\omega\omega}&=-28a_0^2I+16\mathbf a\mathbf a^T
                                      -49\mathbf b\mathbf b^T.&&\nonumber
\end{align}
The transpose of every displayed off-diagonal block is included.

Let \(K_g:\im\HH\to\mathfrak{sp}(2)\) be
\(K_g(v)=\Ad_{g^{-1}}\diag(0,v)\). Define
\begin{equation}\label{eq:two-parameter-metric}
 \widehat P_{\varepsilon,t}(g)
 =\left[(P_0+\varepsilon H_N+tH+t^2J)^{-1}
                          +10K_gK_g^T\right]^{-1}.
\end{equation}
The inner product upstairs is
\(Q(\widehat P_{\varepsilon,t}(g)X,Y)\) on vectors \(gX,gY\).
Its Riemannian quotient under \eqref{eq:GM-action} is denoted by
\(g_{\varepsilon,t}\). In particular, \(g_{0,0}\) is obtained from
\(g_{\rm EK}\) by the left-lower Cheeger deformation with parameter
10. We use \(g_{\rm EK}\) before that deformation and \(g_{0,0}\)
after it. The family in Theorem~\ref{thm:main} is
\begin{equation}\label{eq:final-metric}
 g_\varepsilon=g_{\varepsilon,\varepsilon^3},\qquad
 \widehat P_\varepsilon
 =\left[(P_0+\varepsilon H_N+\varepsilon^3H+\varepsilon^6J)^{-1}
                          +10K_gK_g^T\right]^{-1}.
\end{equation}
Put \(E_{KH}=\Sp(2)\times K\times H_\Delta\). The first two quotient
maps compose to
\[
 q_{KH}:E_{KH}\longrightarrow\Sp(2),\qquad
 q_{KH}(g,k,h_0)=gk^{-1}h_0^{-1}.
\]
On the common product source \(E=E_{KH}\times L\), with \(L=\Sp(1)\),
define \(q_0=q_{KH}\circ\operatorname{pr}_{E_{KH}}\).
Thus \(q_0(g,k,h_0,l)=gk^{-1}h_0^{-1}\), and the full projection is
\[
 \pi_E(g,k,h_0,l)=[\diag(1,l)gk^{-1}h_0^{-1}]\in\Sigma.
\]
Thus the last factor implements the left-lower Cheeger deformation.
The unperturbed factor metrics are
bi-invariant, with their scales specified in \eqref{eq:product-metric}.
\begin{lemma}\label{lem:smooth-metric}
For all sufficiently small \(\varepsilon,t\),
\eqref{eq:two-parameter-metric} defines a smooth positive-definite metric
operator. The action \eqref{eq:GM-action} and left multiplication by
\(\diag(1,q)\) are isometric. The quotient metric
\(g_{\varepsilon,t}\) depends smoothly on the parameters and the point of
the fixed exotic smooth manifold \(\Sigma\).
\end{lemma}
\begin{proof}
All entries of \(H_N,H,J\) are polynomials in the real components of
\(a,b\). Under \eqref{eq:GM-action}, both first-row quaternions are
conjugated by \(q\); the imaginary coordinate blocks transform by
\(\Ad_q\). The dot products, cross products and adjoint operations in
\eqref{eq:annihilator}--\eqref{eq:J} are equivariant. Left multiplication
by \(\diag(1,q)\) fixes the first row and the left-trivialized tangent
coordinates. The positive operator \(P_0\) has a uniform lower bound,
so compactness gives positivity of its small perturbations. Adding
\(10K_gK_g^T\) to the inverse preserves positivity and smoothness,
without any assumption about the rank of an orbit. This is precisely
Cheeger deformation for the left-lower action; the inverse-operator
formula follows by solving its horizontal lift, as in
\cite[Proposition 1.1]{Z09}. The free GM action then
gives a smooth family of quotient metrics.
\end{proof}

\section{The zero-curvature planes}\label{sec:zero-families}

The relevant zero set lies in \(\Gr_2(T\Sigma)\), a compact
seventeen-dimensional manifold. We first describe it before the final
left-lower Cheeger deformation. For quaternionic calculations write
\[
 P(x)=\begin{pmatrix}0&-\bar x\\x&0\end{pmatrix},\qquad
 D(y,z)=\diag(y,z),\quad
 \mathsf H(y)=D(y,y),\quad Q_-(y)=D(y,-y).
\]
Let \(\mathfrak p,\mathfrak q,\mathfrak h\) be the off-diagonal,
diagonal anti-diagonal and diagonal-diagonal subspaces, with
\(\mathfrak k=\mathfrak q\oplus\mathfrak h\). Use the weighted
variables
\(U=\Phi X,V=\Phi Y\), where
\(\Phi X=X_{\mathfrak p}+X_{\mathfrak q}+\tfrac12X_{\mathfrak h}\).
The classification in~\cite[Lemma 3.1 and Section 4]{EK08} leaves two
possible horizontal zero-plane forms:
\begin{align}
 \Afamily\text{: }&\quad U=P(x),\quad V=D(y,z),\quad z=xyx^{-1},
 \label{eq:A-zero-form}\\
 \Bfamily\text{: }&\quad U=P(x)+\mathsf H(y),\quad V=Q_-(y),
               \quad x,y\in\im\HH,\quad x\perp y.
 \label{eq:B-zero-form}
\end{align}
In \(\Afamily\), \(x,y\ne0\); in \(\Bfamily\), both vectors are nonzero. The remaining
types in the upstairs classification are nowhere horizontal.

Let \(\mathcal G=\Sp(2)\times\Gr_2(\mathfrak{sp}(2))\), and let
\(\mathcal G_{\rm hor}\) consist of the planes horizontal for the
GM submersion with metric \(P_0\). These spaces have dimensions 26
and 20. The free GM quotient identifies
\(\mathcal G_{\rm hor}/\Sp(1)\) with \(\Gr_2(T\Sigma)\).
Distances on the latter bundle use a fixed auxiliary metric.

\begin{lemma}\label{lem:EK-incidence}
The horizontal \(\Afamily\) and \(\Bfamily\) forms define disjoint compact embedded
submanifolds \(\widehat{\mathcal Z}_{\Afamily},\widehat{\mathcal Z}_{\Bfamily}\)
of \(\mathcal G_{\rm hor}\), of dimensions nine and eight.
Their GM quotients \(\ZA^{\rm EK},\ZB^{\rm EK}\) have dimensions
six and five and constitute the zero set of \(g_{\rm EK}\).
\end{lemma}

\begin{proof}
On \(\Afamily\) normalize \(|x|=|y|=1\).
Over \(g\in\Sp(2)\) with first row \((a,b)\), horizontality is
\begin{equation}\label{eq:A-horizontal}
 \im(bx\bar a)=0,\qquad
 ay\bar a+bz\bar b=y+z,\qquad z=xyx^{-1}.
\end{equation}
Take the solution set in \(\Sp(2)\times S^3\times S^2\) and divide
by the independent signs of \(x,y\). It is compact. If \(ab\ne0\),
write \(a=rA_0,b=sB_0\), with \(r,s>0\) and unit quaternions
\(A_0,B_0\). The first equation gives \([x]=[B_0^{-1}A_0]\);
its differential in the projective \(x\) variable has rank three.
Consequently \(w=\Ad_{A_0}y=\Ad_{B_0}z\), and the second equation
becomes
\[
 w-y-z=0,\qquad y=Aw,\quad z=Bw,\qquad
 A=\Ad_{A_0^{-1}},\quad B=\Ad_{B_0^{-1}}.
\]
The variables \((A,B,w)\) are independent rotation and unit-vector
coordinates before imposing this last equation. Since
\(|w|=|y|=|z|=1\) and \(w=y+z\), one has
\(\ip yz=-1/2\); in particular \(y,z\) are nonparallel.
Rotation variations therefore have image
\((Aw)^\perp+(Bw)^\perp=\R^3\). Thus the six horizontal equations
are transverse also when \(\dim\ker(I-A-B)=2\).

At \(b=0\), the first equation has differential
\(\dot b\mapsto\im(\dot b\,x\bar a)\), of rank three. The second
is \(Ry-y-z=0\), where \(R=\Ad_a\). The same unit-length identity
gives \(\ip yz=-1/2\). Variations of \(a,x\) give
\((Ry)^\perp+z^\perp=\R^3\), since \(Ry=y+z\) is not parallel
to \(z\). These variations leave the first equation unchanged at
the pole, so the six rows still have full rank. The case \(a=0\) is
the same. The \(\Afamily\) incidence upstairs
therefore has dimension \(10+3+2-6=9\).

For \(\Bfamily\) normalize \(|y|=1\), use \(x\in y^\perp\setminus0\), and
divide by \((x,y)\sim(-x,-y)\). Its horizontal equations are
\begin{equation}\label{eq:B-horizontal}
 ay\bar a=by\bar b,\qquad
 2\im(bx\bar a)+\tfrac12(ay\bar a+by\bar b-2y)=0.
\end{equation}
The first is equivalent to \(|a|=|b|=1/\sqrt2\) and
\([a^{-1}b,y]=0\). Its differential has rank three, also when
\(a^{-1}b=\pm1\): after conjugation by \(a^{-1}\), row-norm variations
supply \(\R y\) and rotations supply \(y^\perp\). The second equation solves
for the two components of \(x\) in \(y^\perp\); the remaining
condition is
\(\ip{\Ad_{a^{-1}}y}{y}=1/2\), whose rotational differential is
nonzero. It also gives \(|x|=\sqrt3/2\). The solution space is thus
compact and smooth of dimension \(10+4-6=8\).

Apply \(\Phi^{-1}\) to the planes in the two normal forms to obtain
their points in \(\mathcal G_{\rm hor}\). These maps are embeddings.
For \(\Afamily\), the pure \(\mathfrak p\)-line and pure \(\mathfrak k\)-line of
\eqref{eq:A-zero-form} recover the two projective parameters. Their
differentials recover the parameters as well: a stationary plane forces
\(\dot x\in\R x\) and then \(\dot y\in\R y\), so the unit
normalizations force both to vanish. For \(\Bfamily\), intersection with
\(\mathfrak k\) first recovers \([y]\); the other line then recovers
\(x\). Infinitesimally \(\dot y=0\), after which a pure
\(\mathfrak p\) variation lying in the \(\Bfamily\) plane must vanish. The
families are disjoint, since \(\Afamily\) has a pure \(\mathfrak p\)-line and \(\Bfamily\)
does not. The free GM quotient reduces both dimensions by three and
preserves their compact embedded character. The classification cited
above shows that these two families exhaust the zero~set.
\end{proof}

\begin{lemma}\label{lem:EK-normal-gap}
There are a neighborhood of \(\ZA^{\rm EK}\cup\ZB^{\rm EK}\)
and a constant \(c_{\rm EK}>0\) on which
\[
 \sec_{g_{\rm EK}}(\sigma)\ge
 c_{\rm EK}\dist(\sigma,\ZA^{\rm EK}\cup\ZB^{\rm EK})^2.
\]
The normal curvature Hessian has rank eleven on \(\ZA^{\rm EK}\)
and rank twelve on \(\ZB^{\rm EK}\).
\end{lemma}
\begin{proof}
Choose a smooth local frame for each plane and consider the five bracket defects
\begin{equation}\label{eq:five-defects}
 [U,V],\quad[U_{\mathfrak k},V_{\mathfrak k}],\quad
 [X_{\mathfrak p},Y_{\mathfrak p}],\quad
 [X_{\mathfrak q},Y_{\mathfrak q}],\quad
 [X_{\mathfrak h},Y_{\mathfrak h}].
\end{equation}
Write \(\mathcal D\) for this tuple. A frame change multiplies all
five components by its determinant, so the zero set and the kernel of
\(d\mathcal D\) at a zero are intrinsic to the plane.
The two product-submersion formulas quantitatively bound the
\(\Sp(2)\) curvature below by a positive constant times the sum of their squared
norms; the triangular comparison is given in
Appendix~\ref{app:clean}. The joint linearization has rank eleven
on \(\Afamily\) and twelve on \(\Bfamily\). Indeed, on \(\Afamily\) write
\(\dot U=P(\xi)+D(p,q)\),
\(\dot V=P(\chi)+D(r,s)\). The second and third defects impose
\(p=\alpha y,q=\beta z,\chi\in\R x\), of total rank seven.
The next two impose \(\alpha=\beta\), because \([y,z]\ne0\).
The remaining defect is
\(\xi y-z\xi+xr-sx\), of rank three.
For \(\Bfamily\) the corresponding triangular ranks are
\(3+2+2+2+3=12\); details are in Appendix~\ref{app:clean}.
The kernels on pairs of independent vectors have dimensions nine and
eight. The tangent variations of the normal forms, together with the
four changes of basis, have precisely these dimensions and span the kernels.
These are the ranks before the horizontal constraint is imposed.
The six horizontal equations are transverse when restricted to the
corresponding bracket-zero family, by Lemma~\ref{lem:EK-incidence}.
It follows that
\[
 \ker\!\left((d\mathcal D)_z|_{T_z\mathcal G_{\rm hor}}\right)
      =T_z\widehat{\mathcal Z}_\bullet,
 \qquad z\in\widehat{\mathcal Z}_\bullet,\quad \bullet\in\{\Afamily,\Bfamily\}.
\]
Equivalently, in the bundle of independent frames the four directions
of a change of basis lie in the kernel; quotienting them leaves the
same normal rank. The three free GM orbit directions also lie in this
kernel. Thus the normal ranks on the seventeen-dimensional quotient
are eleven and twelve. On each compact zero family the least nonzero
singular value of the normal differential is bounded away from zero.
Taylor expansion of the defects and the bracket bound in
Appendix~\ref{app:clean} give the stated uniform inequality.
\end{proof}

For the final Cheeger deformation, let \(\mathcal K\) be the
infinitesimal left-lower action on \((\Sigma,g_{\rm EK})\), and let
\(\mathcal K^*\) denote its adjoint for that metric and
\(Q_L=|\cdot|^2\) on \(\im\HH\). At deformation
parameter \(T=10\), put
\(C_T=(I+T\mathcal K\mathcal K^*)^{-1}\), and define
\[
 \ZA=C_T^{-1}\ZA^{\rm EK},\qquad \ZB=C_T^{-1}\ZB^{\rm EK}.
\]
Here the inverse acts on the plane at its fixed base point.

\begin{proposition}\label{prop:clean-zero}
The zero set of \(g_{0,0}\) is the disjoint union of compact embedded
submanifolds \(\ZA\) and \(\ZB\) of dimensions six and five. Every
plane in either family lies in an immersed totally geodesic flat.
There are a neighborhood \(\mathcal U\) of their union and a constant
\(c_0>0\) such that
\begin{equation}\label{eq:clean-gap}
 \sec_{g_{0,0}}(\sigma)
 \ge c_0\dist(\sigma,\ZA\cup\ZB)^2
 \qquad(\sigma\in\mathcal U).
\end{equation}
Their normal curvature Hessians are positive definite.
\end{proposition}
\begin{proof}
In the standard Cheeger coordinates, obtained by inverting the L
coordinate of E, the projection is \((p,l)\mapsto l^{-1}\cdot p\).
The horizontal lift of a vector \(X\) is
\((C_TX,-T\mathcal K^*C_TX)\) in the product with group metric
\(T^{-1}Q_L\); see also~\cite[Equation (1.2)]{Z09}.
Thus \(\sigma\mapsto C_T\sigma\) is a smooth invertible map of
the Grassmann bundle. The product curvature and O'Neill's
formula~\cite{O66} give
\[
 \sec_{g_{0,0}}(\sigma)
 \ge c_T\sec_{g_{\rm EK}}(C_T\sigma)
\]
for some \(c_T>0\): the ratio of the two Gram determinants is
smooth and positive on the compact Grassmann bundle. The commuting horizontal
lifts in Sections~\ref{sec:B-bound} and~\ref{sec:A-source} have first
three components lifting the \(\Afamily\) and \(\Bfamily\) forms above; their L components
impose the last horizontality condition. Hence both zero families are
retained. The comparison excludes all other zeros. Since \(C_T\) and
its inverse are uniformly Lipschitz on the compact Grassmann bundle,
it also transfers the quadratic estimate of Lemma~\ref{lem:EK-normal-gap}.

The commuting pairs exponentiate to totally geodesic flats in the
bi-invariant product E. Horizontality persists on each flat: for a
vertical generator with left and right components \(\ell,r\), at a
source point \(p\in E\) the
moment \(Q_E(U,\Ad_{p^{-1}}\ell-r)\) has zero derivative in either
commuting direction \(U,V\), by Ad-invariance of \(Q_E\). The same
holds for the moment of V. The flat therefore projects isometrically
and totally geodesically, as in~\cite[Theorem 1.1]{T09}.
\end{proof}

\section{The perturbation argument}\label{sec:completion}

View the sectional curvature functions \(F_\tau\) of a smooth metric
path on the fixed Grassmann bundle \(\mathcal F=\Gr_2(T\Sigma)\).
Use the fixed auxiliary metric to define its normal bundles. A point of \(\mathcal F\) specifies both a base point
and a two-plane; we call it a \emph{flag}. Suppose \(Z\)
is a compact embedded family of critical minima of \(F_0\):
\(F_0|_Z=0\), \(dF_0|_Z=0\), and its Hessian \(\mathcal H\) is
positive definite on the normal bundle of \(Z\). Proposition~\ref{prop:clean-zero}
establishes these properties for \(g_{0,0}\). Assume that the first
variation of sectional curvature vanishes on \(Z\):
\[
 L_1:=\left.\partial_\tau F_\tau\right|_{\tau=0},\qquad L_1|_Z=0.
\]
For \(z\in Z\), put
\[
 b_z=\left.dL_1\right|_{\nu_zZ},\qquad
 a_z=\left.\partial_\tau^2F_\tau(z)\right|_{\tau=0},\qquad
 \mathcal H_z=\left.\operatorname{Hess}F_0\right|_{\nu_zZ}.
\]
We identify normal covectors and vectors using the auxiliary metric.
Define
\begin{equation}\label{eq:reduced-variation}
 \calQ(z)=\inf_{\dot\sigma\in T_z\Gr_2(T\Sigma)}
 \left.\frac{d^2}{d\tau^2}F_\tau(\sigma(\tau))\right|_{\tau=0},
 \qquad \sigma(0)=z,\quad\sigma'(0)=\dot\sigma.
\end{equation}
Because \(dF_0=0\), the expression depends only on the velocity, not
the acceleration. The tangent directions to \(Z\) make no contribution,
since both \(dF_0\) and \(L_1\) vanish along \(Z\). For a normal
velocity \(v\), the second derivative is
\(a_z+2b_z^Tv+v^T\mathcal H_zv\). Consequently
\(\calQ(z)=a_z-b_z^T\mathcal H_z^{-1}b_z\).

The loss \(b^T\mathcal H^{-1}b\) measures the effect of moving the
flag towards the minimum of the perturbed curvature. For example,
\(F_\tau(v)=v^2-4\tau v+\tau^2\) is positive at the fixed point
\(v=0\), but its value at \(v=2\tau\) is \(-3\tau^2\).
The same issue occurs when a plane or its base point moves with the
metric. The required estimate is therefore a lower bound for the
minimum over all such velocities.

\begin{lemma}\label{lem:normal-minimum}
If \(\calQ\ge q_0>0\) on \(Z\), there are a fixed neighborhood
\(\mathcal U_Z\), constants \(c,\tau_0>0\), and a uniform bound
\(F_\tau\ge c\tau^2\) on \(\mathcal U_Z\) for
\(0<|\tau|<\tau_0\). The constants can be chosen uniformly in any
additional compact family for which the stated hypotheses are uniform.
\end{lemma}
\begin{proof}
Use normal coordinates \((z,\nu)\) to \(Z\). Uniform Taylor
expansion gives
\[
 F_\tau(z,\nu)
 =\tfrac12\ip{\mathcal H_z\nu}{\nu}
   +\tau\ip{b_z}{\nu}+\tfrac12\tau^2a_z
   +O(|\nu|^3+|\tau||\nu|^2+\tau^2|\nu|+|\tau|^3).
\]
The normal Hessian remains uniformly positive on a fixed small tube.
The implicit-function theorem gives its unique normal minimum at
\(\nu_\tau=-\tau\mathcal H_z^{-1}b_z+O(\tau^2)\), where
\(F_\tau(z,\nu_\tau)=\tfrac12\tau^2\calQ(z)+O(|\tau|^3)\).
Shrinking the common parameter interval gives the assertion. The same
argument with the additional parameters retained is uniform by compactness.
\end{proof}

\begin{lemma}\label{lem:preserved-flat}
Let \(T\) be an immersed totally geodesic flat for \(g\), and suppose
a smooth bundle map \(D\) vanishes on \(TT\). For a smooth symmetric
operator \(S\), the tensor \(h=D^*SD\) preserves the induced metric
and total geodesy of \(T\) under \(g+\tau h\), whenever this is a
metric. Every plane tangent to a totally geodesic flat is a critical
flag for sectional curvature.
\end{lemma}
\begin{proof}
For tangent \(U,V\) and arbitrary \(W\), one has \(h(U,W)=0\)
and \((\nabla_Wh)(U,V)=0\), since both factors \(D\) vanish on the
tangent arguments. Tangential differentiation and total geodesy also
give \((\nabla_Uh)(V,W)=0\). The exact connection-difference formula
therefore vanishes on tangent pairs, proving preservation.

For the last assertion, Gauss and Codazzi give
\(R(U,V)U=R(U,V)V=0\) along the flat. This annihilates every derivative
in a plane direction. The second Bianchi identity expresses
\((\nabla_NR)(U,V,V,U)\) as tangential derivatives of the same zero
curvature components. Take \(U,V\) parallel along the flat and extend
\(N\) normally; total geodesy keeps all the extra derivatives normal,
and Codazzi annihilates the resulting terms. Thus the point derivatives
also vanish.
\end{proof}

\begin{proposition}\label{prop:A-preservation}
The \(\varepsilon H_N\) perturbation preserves every \(\Afamily\) flat isometrically
and totally geodesically, including its horizontal product-source lift.
For small \(\varepsilon\), \(\Afamily\) consists of exact critical zeros of
\(g_{\varepsilon,0}\) and has uniformly positive normal curvature
Hessian.
\end{proposition}
\begin{proof}
One has \(\norm{\zeta}_Q^2=2m(1-m)\) and
\(\ad_\zeta^T=-\ad_\zeta\). Hence
\begin{equation}\label{eq:N-gram}
 Q(NW,W)=\norm{\zeta\wedge OW}_Q^2+
                         \norm{O[\zeta,W]}_Q^2.
\end{equation}
For nonpolar \(\Afamily\) data, \eqref{eq:A-horizontal} gives
\(\zeta\parallel P(x)\) and
\([\zeta,D(y,z)]=0\). Since \(P_0\) is the identity on
\(\mathfrak p\) and \(P_0\Phi^{-1}D(y,z)=D(y,z)/2\), it follows
that \(NP_0\) kills \(P(x)\) and \(\Phi^{-1}D(y,z)\) in the
\(\Sp(2)\) projection of the \(\Afamily\) flat. At a row pole,
\(\zeta=N=0\), so the same statement holds without division by a row
entry. Equation~\eqref{eq:HN} has the required two factors of \(NP_0\).

To check annihilation along the whole flat, write
\(a=rA,b=sB\), where \(r^2+s^2=1\), \(|A|=|B|=1\), and
\(x=\bar BA\). At a row pole this last relation determines the
otherwise unused unit quaternion. Put \(w=y+z\). The horizontal
identities become \(\Ad_Ay=\Ad_Bz=w\), and the projection of the
source flat is
\[
 \psi(\xi,\eta)=g\exp(\xi P(x))\exp(\eta D(y,z))
                                  \exp(\eta\mathsf H(w/2)).
\]
Set
\[
 \begin{gathered}
 q_\eta=\exp(\eta w/2),\qquad
 r_\xi=r\cos\xi+s\sin\xi,\quad
 s_\xi=s\cos\xi-r\sin\xi,\\
 A_\eta=q_\eta^2Aq_\eta,\qquad B_\eta=q_\eta^2Bq_\eta,\qquad
 (x_\eta,y_\eta,z_\eta)=\Ad_{q_\eta^{-1}}(x,y,z).
 \end{gathered}
\]
Since \(P(x)\) commutes with \(D(y,z)\), the first row of \(\psi\)
is \((r_\xi A_\eta,s_\xi B_\eta)\). Moreover,
\[
 \Ad_{A_\eta}y_\eta=\Ad_{B_\eta}z_\eta=w=y_\eta+z_\eta,
 \qquad z_\eta=x_\eta y_\eta x_\eta^{-1}.
\]
The left-trivialized coordinate tangent fields are
\(P(x_\eta)\) and \(\Phi^{-1}D(y_\eta,z_\eta)\).
At this point \(\zeta=-r_\xi s_\xi P(x_\eta)\), so the preceding
calculation annihilates both tangent fields for every \(\xi,\eta\),
including crossings of a row pole.

On the product source, therefore, the bundle map \(NP_0\,dq_0\)
annihilates the tangent bundle of the entire flat, with
\(q_0:E\to\Sp(2)\) as defined in Section~\ref{sec:metric}.
The pulled-back perturbation pairs a flat
tangent vector with every source vector to zero. Its pairing with the
vertical space is unchanged, so the flat remains horizontal.
Apply Lemma~\ref{lem:preserved-flat} on the source and then take the
remaining quotients.
The planes stay full critical zeros. Their normal Hessian remains
uniformly positive by Proposition~\ref{prop:clean-zero}, smoothness and
compactness.
\end{proof}

We record the two estimates proved in the next sections:
\begin{align}
 \calQ_{\Bfamily}(H_N)&>1/48 &&\text{on }\ZB,\label{eq:B-main-bound}\\
 \calQ_{\Afamily}(H,J)&>1/8 &&\text{on }\ZA.\label{eq:A-main-bound}
\end{align}
The second uses the path \(tH+t^2J\) at \(\varepsilon=0\), and
the first uses \(\varepsilon H_N\) at \(t=0\). The first curvature
variation of \(H\) vanishes on \(\Afamily\), as proved in
Lemma~\ref{lem:A-first-null} below.

\begin{proof}[Proof of the main results]
Choose disjoint fixed tubes around \(\ZA\) and \(\ZB\).
Lemma~\ref{lem:normal-minimum} and \eqref{eq:B-main-bound} give
\(F_{\varepsilon,0}\ge c_{\Bfamily}\varepsilon^2\) on the \(\Bfamily\) tube.
Proposition~\ref{prop:A-preservation} gives a nonnegative fixed \(\Afamily\) tube
whose only zeros are \(\Afamily\). The remaining compact set stays positive by
continuity. Thus \(g_{\varepsilon,0}\) is nonnegative for all
sufficiently small \(\varepsilon>0\), with zero set exactly \(\Afamily\).

The first \(t\)-variation on the retained flats is independent of
\(\varepsilon\). Their induced metric and source horizontal lifts are
unchanged; the intrinsic first variation therefore uses the same tensor
restriction. The Gauss term has zero first variation at a totally
geodesic flat. Thus the first \(t\)-variation is still zero on \(\Afamily\).
The reduced second variation and the normal Hessian depend continuously
on \(\varepsilon\). By \eqref{eq:A-main-bound} they remain uniformly
positive on \(\Afamily\) for \(0\le\varepsilon\le\varepsilon_0\), after
shrinking \(\varepsilon_0\). The uniform version of
Lemma~\ref{lem:normal-minimum} gives
\(F_{\varepsilon,t}\ge c_{\Afamily} t^2\) on a fixed \(\Afamily\) tube, with a fixed
admissible interval for \(t\).

On the complement of a smaller \(\Afamily\) tube there is a uniform lower bound
\(F_{\varepsilon,0}\ge c_1\varepsilon^2\): this follows from the \(\Bfamily\)
estimate and the positive minimum on the remaining compact set.
Smoothness gives a uniform bound
\(|F_{\varepsilon,t}-F_{\varepsilon,0}|\le C|t|\).
Choose \(\eta>0\) with \(C\eta\le c_1/2\), and decrease
\(\varepsilon_0\) so that \(0<t\le\eta\varepsilon^2\) stays in the
fixed admissible interval for \(t\). The complement then has curvature
at least \(c_1\varepsilon^2/2\), while the \(\Afamily\) tube has curvature at least
\(c_{\Afamily} t^2\). With \(c=\min\{c_{\Afamily},c_1/2\}\), this proves
Corollary~\ref{cor:parameter-region}.

Decrease \(\varepsilon_0\) further so that
\(\varepsilon_0\le\min\{1,\eta\}\). The path \(t=\varepsilon^3\)
then lies in this region and satisfies
\(\min\{t^2,\varepsilon^2\}=\varepsilon^6\), proving the curvature
bound in Theorem~\ref{thm:main}. Finally,
Lemma~\ref{lem:smooth-metric} gives smooth dependence on the parameters
on the fixed compact \(\Sigma\), so
\(g_{\varepsilon,\varepsilon^3}\to g_{0,0}\) in \(C^\infty\).
\end{proof}

\section{Curvature variation on the product source}\label{sec:source-variation}

For a metric \(g\) with Levi-Civita connection \(\nabla\), we use
\[
 R^g(X,Y)Z=\nabla_X\nabla_YZ-\nabla_Y\nabla_XZ-\nabla_{[X,Y]}Z.
\]
Its curvature numerator is \(\kappa_g(X,Y)=g(R^g(X,Y)Y,X)\).
For linearly independent \(X,Y\),
\[
 \sec_g(X\wedge Y)=\frac{\kappa_g(X,Y)}{\Gamma_g(X,Y)},\qquad
 \Gamma_g(X,Y)=g(X,X)g(Y,Y)-g(X,Y)^2.
\]
We use ordinary Taylor coefficients along a smooth flag curve:
\[
 F_\tau(\sigma_\tau)=F^{[0]}+\tau F^{[1]}+\tau^2F^{[2]}+O(\tau^3).
\]
Thus the second derivative is \(2F^{[2]}\), and the reduced variation
in \eqref{eq:reduced-variation} is \(2\inf_{\dot\sigma}F^{[2]}\).

The unperturbed metric on the product source \(E\) is
\begin{equation}\label{eq:product-metric}
 Q_E=Q_G\oplus Q_K\oplus |\cdot|^2\oplus\tfrac1{10}|\cdot|^2.
\end{equation}
On \(\mathfrak{sp}(2)\) we use the rational frame
\(D(d_1,d_2)+P(\rho+\omega)\). Thus
\(Q_G=\diag(I_6,2I_4)\), and \(Q_E\) has matrix
\(\diag(I_6,2I_4,I_9,I_3/10)\). The relation to
\eqref{eq:orth-frame} is
\begin{equation}\label{eq:frame-conversion}
 \rho_{\rm orth}=-\sqrt2\,\rho_{\rm rat},\qquad
 \omega_{\rm orth}=\sqrt2\,\omega_{\rm rat}.
\end{equation}
The change-of-frame matrix is
\(T_{\rm fr}=\diag(I_6,-\sqrt2,\sqrt2I_3)\); a tensor matrix
transforms by \(B_{\rm rat}=T_{\rm fr}^TB_{\rm orth}T_{\rm fr}\).
In particular, for
\(h=H/\sqrt2\) the rational-frame \(d_2\rho\) block is
\(+4\mathbf b\). This convention will also separate
\(J=J_r+\sqrt2J_i\) into rational tensors.

The first two quotients of \(E\) give
\(q_{KH}\times\operatorname{id}_L:E\to\Sp(2)\times L\).
The induced metric on the \(\Sp(2)\) factor has operator \(P_0\):
its inverse is the identity on the off-diagonal
subspace and has diagonal block
\(\left(\begin{smallmatrix}3I&I\\I&3I\end{smallmatrix}\right)\).
A pullback \(q_0^*h\) kills those quotient fibres and leaves their
horizontal distribution unchanged. Therefore
\[
 Q_E+\varepsilon q_0^*H_N+tq_0^*H+t^2q_0^*J
\]
induces exactly the linear/quadratic \(\Sp(2)\) metric in
\eqref{eq:two-parameter-metric} before its last Cheeger deformation,
together with the \(L\) metric in \eqref{eq:product-metric}.
The remaining Cheeger and GM quotients give \(g_{\varepsilon,t}\).

For the following local calculation, write \(Q\) for the constant
source metric and \(h,j\) for symmetric tensors invariant under the fibre action. All source
derivatives \(D_W\) use its left-invariant frame. Let \(U,V\) be a
horizontal commuting pair. Write \(L,R:\R^{15}\to T E\) for the
left and right infinitesimal fibre actions at the source point, and put
\begin{equation}\label{eq:fibre-data}
 \mathsf V=L-R,\qquad G=\mathsf V^TQ\mathsf V,
 \qquad C=(L+R)^TQ.
\end{equation}
The combined fibre action is free, so \(G\) is positive definite.
A superscript \(T\) denotes ordinary matrix transpose; the metric
adjoint of an endomorphism \(A\) is \(Q^{-1}A^TQ\).

For a curve of quotient flags, choose a smooth local GM representative
in \(\Sp(2)\) and keep the other source point coordinates at the
identity. Lift a basis by orthogonal projection to the horizontal space
of the varying metric. The invertible vertical Gram matrix makes these
lifts smooth. Their first-order displacement has the form
\(\delta=(D,\dot U,\dot V)\in\R^{10+22+22}\), after using the
fibre action to represent the point displacement in the \(\Sp(2)\)
factor. Differentiated metric-dependent horizontality reads
\begin{equation}\label{eq:affine-horizontal}
 \mathsf E\delta+f=0,\qquad
 \mathsf E\delta=
 \binom{\dot{\mathsf V}_D^{\,T}QU+\mathsf V^TQ\dot U}
       {\dot{\mathsf V}_D^{\,T}QV+\mathsf V^TQ\dot V},\qquad
 f=\binom{\mathsf V^ThU}{\mathsf V^ThV}.
\end{equation}
Here \(\dot L_D=-\ad_D L\) in the G factor and \(\dot R_D=0\).
Define the bracket derivative
\(\beta=\mathsf B\delta=[\dot U,V]+[U,\dot V]\).
Every curve of points and planes supplies a velocity satisfying
\eqref{eq:affine-horizontal}. The fifty-four source coordinates include
gauge and frame directions. The matrices \(\mathsf E,\mathsf B\)
have sizes \(30\times54,22\times54\), respectively; \(f\) and the
constraint multiplier \(\lambda\) have thirty components.

Define the connection variation \(\mathcal C_h\) by
\begin{equation}\label{eq:connection-variation}
 2Q\mathcal C_h(A,B)
 =(D_Ah)B+(D_Bh)A-D_\cdot h(A,B)
             -\ad_B^ThA-\ad_A^ThB.
\end{equation}
The covector \(D_\cdot h(A,B)\) has value \(D_Wh(A,B)\) on \(W\).
Set
\begin{align}
 c_h&=\norm{\mathcal C_h(U,V)}_Q^2
             -\ip{\mathcal C_h(U,U)}{\mathcal C_h(V,V)}_Q,
 \label{eq:source-constant}\\
 P_h&=\tfrac12(D_UD_V+D_VD_U)h(U,V)
              -\tfrac12D_U^2h(V,V)-\tfrac12D_V^2h(U,U),
 \label{eq:principal-symbol}\\
 k_h&=\tfrac32\{(D_Uh)V-(D_Vh)U\}
                  +\tfrac12(\ad_V^ThU-\ad_U^ThV),
 \label{eq:source-lower}\\
 \kappa&=dP_h+\mathsf B^Tk_h,
 \label{eq:source-kappa}\\
 \omega_h&=-(\ad_UL)^ThV+(\ad_VL)^ThU
                    +\mathsf V^T\{(D_Uh)V-(D_Vh)U\}.
 \label{eq:source-omega}
\end{align}
The differential in \eqref{eq:source-kappa} is taken in all the point
and vector variables \((D,\dot U,\dot V)\), before restricting to
the zero family. It includes the normal derivatives of \(P_h\), even
when \(P_h\) vanishes on the family.

\begin{proposition}\label{prop:source-quadratic}
Suppose the first curvature variation of \(h\) vanishes at the
horizontal commuting pair. For the source metric \(Q+\tau h\), the
second Taylor coefficient of the quotient curvature numerator along any
actual flag curve, with source velocity \(\delta\), is
\begin{equation}\label{eq:full-source-quadratic}
 \Phi_h(\delta)=c_h+\kappa^T\delta+
        \tfrac14\norm{\mathsf B\delta}_Q^2+
        \tfrac34\norm{\omega_h+C\mathsf B\delta}_{G^{-1}}^2.
\end{equation}
The sectional coefficient is \(\Phi_h/\Gamma\), where
\(\Gamma=\norm{U}_Q^2\norm{V}_Q^2-\ip UV_Q^2\).
For the path \(Q+\tau h+\tau^2j\), it is
\(\Phi_h/\Gamma+L_j\), where \(L_j\) is the first sectional
curvature variation of \(j\) at the fixed flag.
\end{proposition}
\begin{proof}
Koszul's formula, with the background connection
\(\nabla_AB=[A,B]/2\), gives
\eqref{eq:connection-variation}. In normal coordinates for the background
metric, the second-coordinate-derivative part of the lowered curvature
tensor is linear in \(Q+\tau h\). The background Christoffel symbols
vanish at the point. Its fixed-pair quadratic coefficient is therefore
the Christoffel product \eqref{eq:source-constant}.

For arbitrary constant left-invariant vectors \(A,B\), put \(\gamma=[A,B]\).
The curvature-operator variation is
\((\nabla_A\mathcal C_h)(B,B)-(\nabla_B\mathcal C_h)(A,B)\).
Its pairing with \(A\), including the variation of the lowering metric, is
\begin{align*}
 &Q(D_A\mathcal C_h(B,B)-D_B\mathcal C_h(A,B),A)
       -\tfrac12Q(\mathcal C_h(A,B),\gamma)\\
 &\qquad-\tfrac32Q(\mathcal C_h(\gamma,B),A)
       +\tfrac14h(A,[B,\gamma]).
\end{align*}
By \eqref{eq:connection-variation} and \([D_A,D_B]=D_\gamma\),
the first variation of the source curvature numerator is
\begin{align*}
 L_h^{\rm src}(A,B)={}&P_h(A,B)
       +\tfrac32\{(D_Ah)(B,\gamma)-(D_Bh)(A,\gamma)\}\\
       &+\tfrac12\{h(A,[B,\gamma])-h(B,[A,\gamma])\}
          -\tfrac34h(\gamma,\gamma).
\end{align*}
At \((A,B)=(U,V)\), one has \(\gamma=0\) and
\(d\gamma(\delta)=\beta\). The last term has zero differential,
so the full point-and-plane differential is
\(dP_h(\delta)+k_h^T\beta=\kappa^T\delta\).
The bi-invariant background numerator contributes
\(\norm{\beta}_Q^2/4\).

For the submersion term, use the vertical moment forms
\(\alpha_\tau=\mathsf V^T(Q+\tau h)=\alpha_0+\tau\alpha_h\).
For constant left-invariant vectors, \(D_UL=-\ad_UL\) and
\(D_UR=0\). Ad-invariance of \(Q\) therefore gives
\[
 d\alpha_0(U,V)=(L+R)^TQ[U,V]=C[U,V].
\]
For the metric perturbation, the same exterior derivative gives
\begin{align*}
 d\alpha_h(U,V)={}&-(\ad_UL)^ThV+(\ad_VL)^ThU\\
 &+\mathsf V^T\{(D_Uh)V-(D_Vh)U-h[U,V]\}.
\end{align*}
At the commuting pair this is \(\omega_h\). Along a moving pair,
the derivative of \(C[U,V]\) is \(C\beta\), since the derivative
of \(C\) is multiplied by the zero background bracket. Hence
\[
 \left.\frac d{d\tau}d\alpha_\tau(U_\tau,V_\tau)\right|_0
       =\omega_h+C\beta.
\]
For horizontal extensions, \(d\alpha_\tau(U_\tau,V_\tau)\) is
the negative of the vertical-bracket covector. O'Neill's
formula~\cite{O66}, in the normalization also displayed in
\cite[p. 403, Equation (1)]{GM74}, contributes three quarters of its
squared norm. The resulting Taylor coefficient is
\(3\norm{\omega_h+C\beta}_{G^{-1}}^2/4\). Derivatives of the
inverse Gram matrix multiply the zero background covector and do not
contribute at this order. These are all four terms in
\eqref{eq:full-source-quadratic}.
The sectional coefficient is thus \(\Phi_h/\Gamma\), since the
lower-order numerators vanish. The term \(\tau^2j\) adds \(L_j\)
to this coefficient and leaves the first-order horizontal constraint unchanged.
\end{proof}

The expression \(\Phi_h\) is defined on the entire affine space
\(\mathsf E\delta=-f\). Since every actual flag curve supplies a
velocity in this space,
\[
 \inf_{\text{actual flag velocities}} F^{[2]}
 \ge \inf_{\delta:\,\mathsf E\delta=-f}\frac{\Phi_h(\delta)}{\Gamma}
 \qquad\text{for the path }Q+\tau h.
\]
To bound the right-hand side, we express the
linear term as a bracket pairing plus a multiple of the constraint.
A pair \((Z,\lambda)\) realizing this identity is called a
\emph{feasible dual pair}. Completing the two curvature squares will
then remove every occurrence of the velocity from the lower bound.

\begin{lemma}[A feasible dual lower bound]\label{lem:dual}
Suppose
\begin{equation}\label{eq:dual-equations}
 \mathsf B^TQZ+\mathsf E^T\lambda=\kappa.
\end{equation}
Let \(\Pi\) be a \(Q\)-self-adjoint projection satisfying
\(\Pi\mathsf B=\mathsf B\) and \(\Pi Z=Z\). Then, for every
\(\eta\in\R^{15}\) and every feasible velocity \(\delta\),
\begin{equation}\label{eq:dual-lower-bound}
 \Phi_h(\delta)\ge
 c_h-\lambda^Tf-\eta^T\omega_h
       -\norm{Z-\Pi Q^{-1}C^T\eta}_Q^2-\tfrac13\eta^TG\eta.
\end{equation}
\end{lemma}
\begin{proof}
Put \(\rho=QZ-\Pi^TC^T\eta\) and
\(\Omega=\omega_h+C\mathsf B\delta\). Using
\eqref{eq:affine-horizontal} and \eqref{eq:dual-equations} gives
\begin{align*}
 \Phi_h={}&\tfrac14\norm{\mathsf B\delta+2Q^{-1}\rho}_Q^2
     +\tfrac34\norm{\Omega+\tfrac23G\eta}_{G^{-1}}^2\\
     &+c_h-\lambda^Tf-\eta^T\omega_h
               -\rho^TQ^{-1}\rho-\tfrac13\eta^TG\eta.
\end{align*}
The two squares are nonnegative, and
\(Q^{-1}\rho=Z-\Pi Q^{-1}C^T\eta\), proving the claim.
\end{proof}

\section{Curvature near the \texorpdfstring{\(\Bfamily\)}{B} family}\label{sec:B-bound}

On \(\ZB\), the mixed curvature variation factors through the bracket
derivative. The two curvature squares then control it by Cauchy--Schwarz.

\begin{proposition}\label{prop:B-bound}
For the \(\varepsilon H_N\) path and every flag in \(\ZB\),
\[
 \calQ_{\Bfamily}(H_N)\ge\frac18-\frac1{54}
      \left(\frac{50}{9}+\frac{27}{8000}\right)>\frac1{48}.
\]
\end{proposition}

We use the following fixed quaternions, local to this section:
\begin{equation}\label{eq:B-axes}
 y=i+j,\quad x=\frac{i-j-2k}{2},\quad h_0=y/2,\quad
 k_0=y\times x,\quad A_0=(1+i)/\sqrt2,
\end{equation}
and set
\[
 R=\Ad_{A_0^{-1}},\qquad a_0=y/2-x,\qquad
 v_0=3x+9y/2,\qquad m_0=y\times R^Tx.
\]
Direct quaternion multiplication gives
\begin{equation}\label{eq:B-elementary}
 \begin{gathered}
 |y|^2=2,\quad |x|^2=3/2,\quad x\perp y,\quad
 |k_0|^2=3,\quad |m_0|^2=3/4,\\
 Ry=y/2+x,\qquad Rm_0=k_0/2,\qquad
 k_0\times a_0=v_0/3.
 \end{gathered}
\end{equation}

The \(\Bfamily\) horizontal equations imply
\((a,b)=(A,B)/\sqrt2\) with
\(\Ad_Ay=\Ad_By=\eta\), \(|\eta|^2=2\) and
\(\ip y\eta=1\). Simultaneous conjugation puts
\((y,\eta)=(i+j,i+k)\). Thus
\(A=A_0\exp(\alpha y/\sqrt2)\) and
\(B=A_0\exp(\beta y/\sqrt2)\).
Right translation in the G source factor and conjugation in K fix the
source point at
\[
 g_*=\begin{pmatrix}a_*&a_*\\-a_*&a_*\end{pmatrix},
 \qquad a_*=(1+i)/2.
\]
Only the \(H_\Delta\) embedding in K varies: it has the form
\(q\mapsto(Sq,Tq)\), where \(S,T\in\SO(3)\) fix y. This is a
change of source coordinates with transformed fibre action, not an
assertion of right-K invariance of the perturbed quotient metric.
Write \(K_\theta=\diag(q_1,q_2)\), with \(q_1,q_2\) commuting
with y. The projection to \(\Sp(2)\) is
\(q_\theta(g,k,h_0)=gk^{-1}K_\theta h_0^{-1}\). In the fibre
coordinate order \((\xi_1,\xi_2,\xi_h,\xi_\ell,\xi_g)\), its
full source fibre maps are
\begin{align*}
 \mathsf L_f\xi&=(\Ad_{g_*^{-1}}D(\xi_g,\xi_\ell),
                              (S\xi_h,T\xi_h),\xi_g,0),\\
 \mathsf R_f\xi&=(D(\xi_1,\xi_2),(\xi_1,\xi_2),\xi_h,\xi_\ell).
\end{align*}

The horizontal commuting pair is
\begin{equation}\label{eq:B-source-pair}
 \begin{aligned}
 U_G&=D(h_0,h_0)+P(x),& U_K&=(-h_0,-h_0),&
 U_H&=-y,&U_L&=10R^Ta_0,\\
 V_G&=D(h_0,-h_0),& V_K&=(-h_0,h_0),&
 V_H&=0,&V_L&=0.
 \end{aligned}
\end{equation}
Its Gram determinant is 54 for \eqref{eq:product-metric}. Denote the
sum and difference of these maps by \(\mathsf J_f\) and \(\mathsf V_f\).

Write the pulled-back first tensor as \(H_N=fC_N\), where
\(f=|a|^2-1/2\).
\begin{lemma}\label{lem:B-flat-carrier}
On the entire source flat generated by \eqref{eq:B-source-pair},
\begin{equation}\label{eq:B-carrier}
 \begin{gathered}
 f=D_Uf=D_Vf=0,\qquad |df|_Q^2=1/2,\\
 C_N(U,U)=9,\quad C_N(V,V)=-3,\quad C_N(U,V)=0.
 \end{gathered}
\end{equation}
The three restrictions of \(C_N\) are constant in the parallel
coordinate frame of the flat.
\end{lemma}
\begin{proof}
Put \(h_0=y/2\). The projection of the flat to \(\Sp(2)\) is
\[
 \psi_{\Bfamily}(\xi,\eta)=g_*\exp(\xi U_G)\exp(\eta D(y,-y))
                              \exp(3\xi\mathsf H(h_0))K_\theta.
\]
Indeed \(U_G\) commutes with \(V_G\), and the K and \(H_\Delta\)
factors in \eqref{eq:B-source-pair} give the remaining exponentials.
Since \(U_G^2=-2I\) and \(A_0(h_0+x)=yA_0\), its first row is
\[
 \begin{split}
 a(\xi,\eta)&=2^{-1/2}e^{\xi y}A_0e^{(\eta+3\xi/2)y}q_1,\\
 b(\xi,\eta)&=2^{-1/2}e^{\xi y}A_0e^{(-\eta+3\xi/2)y}q_2.
 \end{split}
\]
Both norms are \(1/\sqrt2\), so f vanishes on the entire flat.
The K, \(H_\Delta\), and L coordinates leave \(m\) unchanged.
In the G factor, differentiating the first row gives
\(|dm|_{Q_G}^2=2m(1-m)\); hence \(|df|_Q^2=1/2\) there.

The left-trivialized coordinate tangent fields of \(\psi_{\Bfamily}\) have
the form
\[
 P(\widetilde x)+2\mathsf H(y),\qquad Q_-(y),\qquad
 \widetilde x=q_2^{-1}e^{-3\xi y/2}xe^{3\xi y/2}q_1.
\]
Here \(\widetilde x\) is imaginary, perpendicular to y, and has
squared norm \(3/2\). Also
\(\bar a b=\tfrac12q_1^{-1}e^{-2\eta y}q_2\) lies in
\(\spanof_\R\{1,y\}\). Substitution in \eqref{eq:annihilator}
therefore gives
\[
 N\mathsf H(y)=0,\qquad NP(\widetilde x)=\tfrac12P(\widetilde x),
 \qquad NQ_-(y)=Q_-(y).
\]
The images under \(NP_0\) of the two tangent fields are consequently
\(P(\widetilde x)/2\) and \(Q_-(y)/2\). The operator \(M\) acts by
12 on the first and by \(-3\) on the second. Their squared Q norms
are \(3/4\) and 1, and they are orthogonal. This proves the three
constant restrictions in \eqref{eq:B-carrier}.
\end{proof}

\begin{proof}[Proof of Proposition~\ref{prop:B-bound}]
By Lemma~\ref{lem:B-flat-carrier}, the entire
source tensor \(H_N\), and its derivatives along U and V, vanish at
every \(\Bfamily\) flag. Its restriction to the flat is identically zero, so its
first curvature variation vanishes. Its affine horizontal forcing and \(\omega_{H_N}\)
therefore vanish. The fixed sectional second derivative is
\((27/4)/54=1/8\), by \eqref{eq:source-constant} and
\eqref{eq:B-carrier}.

Let \(\beta=[\dot U,V]+[U,\dot V]\). The background numerator
Hessian on source velocities is
\begin{equation}\label{eq:B-Hessian}
 \mathcal H_{\Bfamily}(\delta,\delta)
 =\tfrac12\norm\beta_Q^2+
          \tfrac32\norm{G^{-1/2}\mathsf J_f^TQ\beta}^2.
\end{equation}
The mixed derivative of the curvature numerator is
\begin{equation}\label{eq:B-scalar-source}
 b_{\Bfamily}(\delta)=\tfrac32\delta(D_U^2f)-\tfrac92\delta(D_V^2f).
\end{equation}
Indeed the principal expression \eqref{eq:principal-symbol}, with
\eqref{eq:B-carrier}, has exactly this derivative. All remaining
linear bracket terms vanish because \(H_N=D_UH_N=D_VH_N=0\).

We next express \eqref{eq:B-scalar-source} through \(\beta\).
Retain arbitrary source variables
\[
 D=D(d_1,d_2)+P(z),\quad
 \dot U_G=D(a_1,a_2)+P(u),\quad
 \dot V_G=D(c_1,c_2)+P(v),
\]
and put \(e=d_1-d_2\), \(r=\re z\), \(C_0=c_1+c_2\),
\(q=\im v\). At \(g_*\), for
\(W=D(r_1,r_2)+P(w)\), direct differentiation of
\(f(g)=\tfrac12(g\diag(1,-1)g^*)_{11}\) gives
\(D_W^2f=\ip{r_1-r_2}{\im w}\).
Differentiating once in D and adding the variations of U,V yields
\begin{equation}\label{eq:B-four-source}
 b_{\Bfamily}=-9r+\tfrac32\ip e{k_0}
           +\tfrac32\ip{a_1-a_2}x-\tfrac92\ip yq.
\end{equation}
Use the identity
\[
 [W,[W,\diag(1,-1)]]=
 \begin{pmatrix}
 -4|w|^2&\overline{2(r_2w-wr_1)}\\
 2(r_2w-wr_1)&4|w|^2
 \end{pmatrix}
\]
and commute once more with D. Its U and V evaluations give
\(\delta(D_U^2f)=-6r+\ip e{k_0}+\ip{a_1-a_2}x\) and
\(\delta(D_V^2f)=\ip yq\), respectively.

Define
\(n_0=(C_0+e\times y)/2-q-r y\) and
\(s_0=S^Tc_1+T^Tc_2\). The differentiated horizontal equations give
\begin{equation}\label{eq:B-horizontal-derivative}
 \begin{gathered}
 \dot U_K=(-a_1,-a_2),\quad\dot V_K=(-c_1,-c_2),\quad
 \dot V_H=-s_0,\quad\dot V_L=10R^Tn_0,\\
 Rs_0=n_0+2q+2r y,\qquad
 \ip{s_0}y=\ip{C_0}y.
 \end{gathered}
\end{equation}
To check the point-dependent terms, the lower and upper left generators
in G at \(g_*\) are
\[
 W(\xi)=D(R\xi/2,R\xi/2)+P(-R\xi/2),\qquad
 A(\xi)=D(R\xi/2,R\xi/2)+P(R\xi/2).
\]
Their derivatives are \(-[D,W]\) and \(-[D,A]\). Pairing with V
adds the moments of
\([D,V_G]=D(d_1\times y,-d_2\times y)
 +P(-\ip{\im z}y+r y)\), which are exactly the
\(e\times y/2\mp r y\) terms in
\eqref{eq:B-horizontal-derivative}. The last scalar identity uses
only \(Sy=Ty=y\).

Write \(\beta_{G,o}\) for the P component of \(\beta_G\) and
\(\beta_{K,\Sigma}=\beta_{K_1}+\beta_{K_2}\). Direct brackets give
\begin{align*}
 \beta_{K,\Sigma}&=(a_1-a_2-C_0)\times y,&
 \im\beta_{G,o}&=(\re u)y+x\times C_0+y\times q,\\
 \beta_H&=2y\times s_0,&
 \beta_L&=200R^T(a_0\times n_0).
\end{align*}
Eliminating e and r in \eqref{eq:B-four-source} gives
\[
 b_{\Bfamily}=\ip{v_0}{n_0}-\tfrac34\ip{k_0}{\beta_{K,\Sigma}}
                         +\tfrac32\ip{k_0}{\im\beta_{G,o}}.
\]
The last G term can also be eliminated. Taking scalar products in
\eqref{eq:B-horizontal-derivative} gives
\[
 \ip{C_0}y=4\ip qx-2\ip{a_0}{n_0},\qquad
 \ip{k_0}{\im\beta_{G,o}}=
                  \ip{m_0}{\beta_H}-\tfrac13\ip{v_0}{n_0}.
\]
Using \(Rm_0=k_0/2\) and \(k_0\times a_0=v_0/3\) in the
last expression for \(\beta_L\), we conclude that
\begin{equation}\label{eq:B-dual-vector}
 b_{\Bfamily}=\ip\lambda\beta_Q,\qquad
 \lambda_G=0,\quad
 \lambda_K=(-3k_0/4,-3k_0/4),\quad
 \lambda_H=3m_0/2,\quad\lambda_L=3m_0/20.
\end{equation}
The L coefficient includes its metric weight \(1/10\).

The same elimination can be checked before imposing the constraints.
Let \(\mathbf b_{\Bfamily}\in\R^{54}\) be the covector in
\eqref{eq:B-four-source}, and form \(\mathsf E,\mathsf B\) from
\(\mathsf L_f,\mathsf R_f\) as in Section~\ref{sec:source-variation}.
Then
\[
 \mathbf b_{\Bfamily}-\mathsf B^TQ\lambda=\mathsf E^T\mu,
\]
where the two halves of the constant multiplier, in the fibre order above,
are
\[
 \begin{split}
 \mu_U&=(-3x/2,\,3x/2,\,0,\,0,\,0),\\
 \mu_V&=(3x/2,\,3x/2,\,0,\,(3,3/2,3/2),\,(-3/2,3/2,-3)).
 \end{split}
\]
Each triple is written in the \((i,j,k)\) basis, and each \(0\)
denotes a zero three-vector.
This identity holds for arbitrary rotations \(S,T\) fixing y and
immediately gives \eqref{eq:B-dual-vector} on \(\ker\mathsf E\).

For every admissible source velocity \(\delta\in\ker\mathsf E\)
and every fibre coefficient \(\eta\), Cauchy--Schwarz in the two
summands of \eqref{eq:B-Hessian} gives
\begin{equation}\label{eq:B-Cauchy}
 b_{\Bfamily}(\delta)^2\le C_\eta\mathcal H_{\Bfamily}(\delta,\delta),\qquad
 C_\eta=2\norm{\lambda-\mathsf J_f\eta}_Q^2
                         +\tfrac23\norm{\mathsf V_f\eta}_Q^2.
\end{equation}
This follows by writing
\(b_{\Bfamily}=\ip{\lambda-\mathsf J_f\eta}{\beta}_Q
 +\ip{G^{1/2}\eta}{G^{-1/2}\mathsf J_f^TQ\beta}\).
Choose the two K coefficients of \(\eta\) to be
\((-k_0/3,-k_0/3)\), its GM left coefficient to be \(3m_0/4\),
and all other coefficients to be zero. The two product vectors are
\[
\begin{array}{c|cc}
 &\lambda-\mathsf J_f\eta&\mathsf V_f\eta\\ \hline
G&D(7k_0/48,7k_0/48)+P(-3k_0/16)
 &D(25k_0/48,25k_0/48)+P(3k_0/16)\\
K&(-5k_0/12,-5k_0/12)&(k_0/3,k_0/3)\\
H&3m_0/4&3m_0/4\\
L&3m_0/20&0
\end{array}
\]
Using \eqref{eq:B-elementary}, the off-diagonal weight two and the
L weight \(1/10\), their squared norms are
\[
 \norm{\lambda-\mathsf J_f\eta}_Q^2=\frac{173}{96}+\frac{27}{16000},
 \qquad \norm{\mathsf V_f\eta}_Q^2=\frac{281}{96}.
\]
Thus \(C_\eta=50/9+27/8000<45/8\). For every
\(\delta\in\ker\mathsf E\), \eqref{eq:B-Cauchy} gives
\[
 \begin{split}
 \mathcal H_{\Bfamily}(\delta,\delta)+2b_{\Bfamily}(\delta)
 &\ge\bigl(\sqrt{\mathcal H_{\Bfamily}(\delta,\delta)}-\sqrt{C_\eta}\bigr)^2-C_\eta\\
 &\ge-C_\eta.
 \end{split}
\]
The complete sectional second variation is therefore bounded below by
\[
 \frac1{54}\{\tfrac{27}{4}+2b_{\Bfamily}+\mathcal H_{\Bfamily}\}
 \ge\frac18-\frac{C_\eta}{54}>\frac1{48}.
\]
The rotations \(S,T\) were arbitrary rotations fixing \(y\), so the estimate
includes the real branch and every angular endpoint.
\end{proof}

\section{Curvature near the \texorpdfstring{\(\Afamily\)}{A} family}\label{sec:A-source}

The \(\Afamily\) calculation uses the rational tensor \(h=H/\sqrt2\).
All source quantities, including the horizontal forcing \(f\), use the
notation of Section~\ref{sec:source-variation} for this tensor.
Set
\begin{equation}\label{eq:A-fixed-quaternions}
 y=i+j,\quad z=-i+k,\quad a_*=y+z=j+k,
 \qquad A_\circ=(1+i+j+k)/2,\quad B_\circ=(1-i+j+k)/2.
\end{equation}
Thus \(|y|^2=|z|^2=|a_*|^2=2\), and
\(\Ad_{A_\circ}y=\Ad_{B_\circ}z=a_*\). For a quaternionic axis
\(e\) of squared length two put
\(q_e(u)=(1-2u^2+2ue)/(1+2u^2)\). Complete \(\Afamily\) flags, modulo the
isometric actions, are represented by
\begin{equation}\label{eq:A-parameters}
 A=A_\circ q_y(u),\quad B=B_\circ q_z(v),\quad
 g=\begin{pmatrix}rA&sB\\-sA&rB\end{pmatrix},\qquad
 r^2+s^2=1,\quad m=r^2.
\end{equation}
Write \(\theta=(u,v)\); both parameters range over the compactified real line.
To see completeness,
the first equation of \eqref{eq:A-horizontal} gives the off-diagonal
line \(x=\bar BA\) up to sign when both first-row entries are nonzero.
The second gives \(\Ad_Ay=\Ad_Bz=y+z\), and hence
\(\ip yz=-1\) in this normalization. Simultaneous conjugation fixes
the ordered pair \((y,z)\) as in \eqref{eq:A-fixed-quaternions}; the
remaining freedoms are exactly the two phase circles. At a row pole,
the off-diagonal plane line determines the unit quaternion of the
vanishing row entry. Thus its phase remains part of the flag data and
is not discarded.

Use the source point and phase matrix
\[
 g_{r,s}=\begin{pmatrix}rA_\circ&sB_\circ\\-sA_\circ&rB_\circ\end{pmatrix},
 \qquad K_\theta=\diag(q_y(u),q_z(v)).
\]
The physical point is \(g_{r,s} K_\theta\), and the projection from
the first three source factors is
\(q_\theta(g,k,h_0)=gk^{-1}K_\theta h_0^{-1}\).
Its left-trivialized differential is
\begin{equation}\label{eq:A-push}
 T(W)=\Ad_{K_\theta^{-1}}(W_G-\diag(W_{K_1},W_{K_2}))
                                  -\diag(W_H,W_H).
\end{equation}
The fixed source pair is
\begin{equation}\label{eq:A-source-pair}
 \begin{split}
 U&=(P(\bar B_\circ A_\circ),0,0,0),\\
 V&=(D(y,z)/2,(-y/2,-z/2),-a_*/2,5a_*).
 \end{split}
\end{equation}
It is horizontal and commuting, and its Gram determinant is \(15\).
\begin{lemma}\label{lem:A-first-null}
The first sectional curvature variation of \(H\) vanishes on the full \(\Afamily\)
family, including both row poles.
\end{lemma}
\begin{proof}
Restrict \(h=H/\sqrt2\) to the full source flat and denote its
coordinate entries by \(h_{11},h_{12},h_{22}\). For the projection
\(\psi(\xi,\eta)\) in Proposition~\ref{prop:A-preservation}, the
blocks in \eqref{eq:H} give
\[
 h_{11}=0,\qquad h_{22}=\tfrac{35}{2}\re a(\xi,\eta).
\]
The coordinate derivatives at the origin are
\[
 \partial_\xi^2h_{22}=-\tfrac{35}{2}\re a,
 \qquad \partial_\xi\partial_\eta h_{12}=-\tfrac{35}{4}\re a.
\]
Appendix~\ref{app:mass-encoding} computes the mixed derivative using
the horizontal \(\Afamily\) identities, without division by a row amplitude.
Substitution in \eqref{eq:J-intrinsic} gives \(L_H=\sqrt2L_h=0\).
\end{proof}

Let \(S_A=\Ad_{q_y(u)}\) and \(S_B=\Ad_{q_z(v)}\). For fibre
coordinates \((\xi_1,\xi_2,\xi_h,\xi_\ell,\xi_g)\), the left and
right maps used in \eqref{eq:fibre-data} are explicitly
\begin{align}
 L\xi&=(\Ad_{g_{r,s}^{-1}}D(\xi_g,\xi_\ell),
                          (S_A\xi_h,S_B\xi_h),\xi_g,0),\label{eq:A-LR}\\
 R\xi&=(D(\xi_1,\xi_2),(\xi_1,\xi_2),\xi_h,\xi_\ell).\nonumber
\end{align}
Since the K and \(H_\Delta\) variables occur in \eqref{eq:A-push}, they also
contribute to the derivatives of the pulled-back tensors.

The projection in Lemma~\ref{lem:dual} can be fixed explicitly. On the
G factor put
\[
 \Pi_U=-\ad_{U_G}^2/4,\qquad
 \Pi_V=-\ad_{V_G}^2/2,\qquad
 \Pi_G=\Pi_U+\Pi_V-\Pi_U\Pi_V.
\]
The other four blocks are
\(I-yy^T/2,I-zz^T/2,I-a_*a_*^T/2,I-a_*a_*^T/2\).
The two G projections commute. These formulas give
\(\Pi^2=\Pi,\ \Pi^TQ=Q\Pi,\ \Pi\mathsf B=\mathsf B\).

Appendix~\ref{app:explicit-dual} specifies one source solution
\((Z^0,\lambda^0)\) and five homogeneous columns
\((K_Z,K_\lambda)\), all with constant scalar divisors. They satisfy
\begin{equation}\label{eq:A-dual-checks}
 \begin{gathered}
 \mathsf B^TQZ^0+\mathsf E^T\lambda^0=\kappa,
 \qquad \mathsf B^TQK_Z+\mathsf E^TK_\lambda=0,\\
 \Pi Z^0=Z^0,\qquad \Pi K_Z=K_Z.
 \end{gathered}
\end{equation}
Write \((\xi,\eta)=\mathcal R x\), where the constant coordinate
map \(\mathcal R:\R^{20}\to\R^5\oplus\R^{15}\) is given in
that appendix. The choices
\[
 Z=Z^0+K_Z\xi,\qquad \lambda=\lambda^0+K_\lambda\xi
\]
satisfy the dual equations for every \(x\). Substituting them in
\eqref{eq:dual-lower-bound} gives a quadratic lower bound in \(x\).
Its matrix, linear term, and constant terms are
\begin{align}
 \mathsf A&=\mathcal R^T
 \begin{pmatrix}
 K_Z^TQK_Z&-K_Z^TC^T\\
 -CK_Z&C\Pi Q^{-1}C^T+G/3
 \end{pmatrix}\mathcal R,\label{eq:A-model}\\
 d&=\mathcal R^T\binom{-K_Z^TQZ^0-K_\lambda^Tf/2}
                              {CZ^0-\omega_h/2},\nonumber\\
 b_R&=c_h-(\lambda^0)^Tf-\norm{Z^0}_Q^2
                                      +\tfrac{15}{2}L_{J_r},\nonumber\\
 b_I&=\tfrac{15}{2}L_{J_i}.\nonumber
\end{align}
Here \(L_J\) denotes the first sectional curvature variation of \(J\).
It is computed intrinsically on the flat: the linearized Gauss term
vanishes because the background flat is totally geodesic.

\begin{proposition}\label{prop:A-mass}
For the full \(\Afamily\) family, the source identities \eqref{eq:A-dual-checks}
hold and the quantities
in \eqref{eq:A-model} have the form
\begin{equation}\label{eq:mass-model}
 \mathsf A=\mathsf A_0+m\mathsf A_1,
 \quad d=s(d_0+md_1),
 \quad b_\bullet=b_{\bullet,0}+mb_{\bullet,1}+m^2b_{\bullet,2},
\end{equation}
where \(\bullet\in\{R,I\}\), all coefficients depend only on
\(u,v\), and \(\mathsf A_1\)
is constant. The linear term \(d\) vanishes on the full \(s=0\) pole family.
\end{proposition}
\begin{proof}
We reduce the assertion to polynomial identities, whose exact verification
is given in Section~\ref{sec:verification}.
Use \(r=(1-w^2)/(1+w^2)\), \(s=2w/(1+w^2)\) and work in
\(\mathbb Q(u,v,w)\). Form the metric jets through
\eqref{eq:A-push}, and differentiate \eqref{eq:principal-symbol}
in the complete source variables before substituting
\eqref{eq:A-source-pair}. Its value on the family is zero. The bracket
and O'Neill terms in the first curvature variation also vanish at a
horizontal commuting pair. The source solution in
Appendix~\ref{app:explicit-dual} is then substituted into all fifty-four
equations of \eqref{eq:A-dual-checks}, as are all five homogeneous
columns. The residual numerators are zero polynomials.

The coefficients in \eqref{eq:mass-model} can be constructed using
\(m=0,1,9/25\); explicit interpolation formulas appear in
Appendix~\ref{app:mass-encoding}. This only defines candidates. The
calculation then reconstructs the source at an independent symbolic \(w\)
and subtracts \eqref{eq:mass-model}. Every matrix, vector and scalar
residual is identically zero, and the slope matrix is angle-independent.
The \(s=0\) source is separately checked before division by any row
factor.
\end{proof}

For the actual path \(tH+t^2J\), let \(a_{\Afamily}\) be the infimum of the
second Taylor coefficient over all actual flag velocities, and define
the relaxed coefficient
\[
 \underline a_h:=\inf_{\delta:\,\mathsf E\delta=-f}
                         \frac{\Phi_h(\delta)}{15}.
\]
Under \(H=\sqrt2h\), the substitution \(\delta=\sqrt2\,\widetilde\delta\)
rescales the linear horizontal forcing by \(\sqrt2\) and the quadratic
coefficient by 2. Proposition~\ref{prop:source-quadratic} therefore gives
\[
 a_{\Afamily}\ge 2\underline a_h+L_{J_r}+\sqrt2L_{J_i}.
\]
Only this lower bound on actual velocities is needed. Lemma~\ref{lem:dual} and
\eqref{eq:A-model}--\eqref{eq:mass-model} imply, for every
\(x\in\R^{20}\),
\begin{equation}\label{eq:A-witness-inequality}
 \tfrac{15}{2}a_{\Afamily}\ge b_R(m)+\sqrt2b_I(m)
       +2s(d_0+md_1)^Tx-x^T(\mathsf A_0+m\mathsf A_1)x.
\end{equation}
\subsection{A polynomial choice of dual variables}

Set \(m_*=1/3\), and partition
\[
 \mathsf A_*:=\mathsf A_0+\tfrac13\mathsf A_1
  =\begin{pmatrix}B_*&L_*\\L_*^T&D_*\end{pmatrix},\qquad
 d_*:=d_0+\tfrac13d_1=\binom{d_B}{d_R}.
\]
The block \(B_*\) is a constant invertible rational \(16\times16\)
matrix; the exact checks are included in Section~\ref{sec:verification}.
Since \(\mathsf A_*\)
is the matrix of a sum of squares in \eqref{eq:dual-lower-bound},
\(B_*\) is positive definite. Define
\begin{equation}\label{eq:small-witness}
 \begin{split}
 S&=D_*-L_*^TB_*^{-1}L_*,\qquad
 b_*=d_R-L_*^TB_*^{-1}d_B,\\
 v_*&=\tfrac1{16}(34I-21S+4S^2)b_*,\qquad
 x_0=\binom{B_*^{-1}(d_B-L_*v_*)}{v_*},\qquad x=sx_0.
 \end{split}
\end{equation}
For fixed \(v_*\), the first sixteen coordinates maximize
\(2d_*^Tx_0-x_0^T\mathsf A_*x_0\). Both \(v_*\) and \(x_0\) are rational functions
of the two angles. The fixed matrix polynomial constructs a feasible
dual choice without introducing an angle-dependent matrix inverse.
Its adequacy is established by the five inequalities below; no inverse
approximation or optimality property is used. We use \(x=sx_0\) at the
actual parameter \(m\).
Since \(d=s(d_0+md_1)\) and \(s^2=1-m\), the resulting gain in
\eqref{eq:A-witness-inequality} is quadratic.

\subsection{The five angular inequalities}

To keep the sign calculation rational, insert
\eqref{eq:small-witness} and form the auxiliary polynomial
\begin{equation}\label{eq:mass-polynomial}
 \begin{split}
 P_\theta(m)={}&b_R(m)+\tfrac75b_I(m)-\tfrac{15}{16}\\
 &+(1-m)\{2(d_0+md_1)^Tx_0
                         -x_0^T(\mathsf A_0+m\mathsf A_1)x_0\}.
 \end{split}
\end{equation}
The use of \(7/5\) in place of \(\sqrt2\) introduces an error that
we estimate below after proving positivity of this polynomial.
It is quadratic in \(m\). Write
\[
 P_\theta(m)=B_0(1-m)^2+2B_1m(1-m)+B_2m^2.
\]
Its degree-four Bernstein coefficients are
\begin{equation}\label{eq:five-coefficients}
 B_0,\quad\frac{B_0+B_1}{2},\quad
 \frac{B_0+4B_1+B_2}{6},\quad\frac{B_1+B_2}{2},\quad B_2.
\end{equation}
The five weights \(\binom4j m^j(1-m)^{4-j}\)
are nonnegative and sum to one. Strict positivity of these five angle
functions therefore implies \(P_\theta(m)>0\) for every
\(0\le m\le1\). This change of basis leaves the polynomial unchanged
and can give positive coefficients even when its original middle
coefficient \(B_1\) is negative.

\begin{proposition}\label{prop:A-sign}
Every function in \eqref{eq:five-coefficients} is strictly positive on
the complete angular family, including both projective endpoints.
Consequently
\[
 a_{\Afamily}>\frac{37}{560}>\frac1{16},\qquad
 \calQ_{\Afamily}(H,J)=2a_{\Afamily}>\frac18.
\]
\end{proposition}
\begin{proof}
The five functions have positive denominators of the form
\(c(u^2+1/2)^a(v^2+1/2)^b\), with \(c>0\). Exact coefficient
comparison verifies invariance under either antipodal map
\[
 u\longmapsto-\frac1{2u},\qquad v\longmapsto-\frac1{2v}.
\]
If \(|u|>3/4\), its image has absolute value less than \(2/3\),
and infinity maps to zero; the same argument applies to \(v\).
Invariance under these two involutions therefore reduces positivity on
the complete projective angular domain to positivity on the closed square
\(\mathcal B=[-3/4,3/4]^2\).
The numerators are reconstructed from \eqref{eq:A-model} and
\eqref{eq:small-witness}; their Bernstein coefficients are strictly
positive on the complete fifteen-rectangle cover specified in
Section~\ref{sec:verification}. This proves the first assertion.

For the irrational part, set
\[
 T(u)=\frac{1+16u-12u^2-32u^3+4u^4}{(1+2u^2)^2}.
\]
For \(J_i\), formula~\eqref{eq:J-intrinsic} gives
\begin{equation}\label{eq:I-formula}
 I:=L_{J_i}=\frac{11}{160}(1-2m)\{3T(u)+7T(v)\}.
\end{equation}
If \(c=1-12u^2+4u^4\), \(s=4u-8u^3\), and
\(d=(1+2u^2)^2\), then
\(c^2+2s^2=d^2\) and
\(9d^2-(c+4s)^2=2(2c-s)^2\). Hence \(|T|\le3\) and
\(|I|\le33/16\), including infinity by continuity. Combining
\eqref{eq:A-witness-inequality}, \(P_\theta>0\), and
\(7/5<\sqrt2<10/7\) yields
\[
 a_{\Afamily}>\frac18-\Bigl(\frac{10}{7}-\frac75\Bigr)\frac{33}{16}
       =\frac{37}{560}>\frac1{16}.
\]
Since \(\calQ_{\Afamily}=2a_{\Afamily}\), this proves \eqref{eq:A-main-bound}.
\end{proof}

\clearpage
\section{Exact computations}\label{sec:verification}

The computation supplies three specific inputs to the proof. The unperturbed
metric has two families of zero-curvature planes, \(\Afamily\) and \(\Bfamily\), and
Section~\ref{sec:completion} reduces the remaining argument to the uniform
bounds \(\calQ_{\Afamily}(H,J)>1/8\) and \(\calQ_{\Bfamily}(H_N)>1/48\).
These reduced variations allow both the base point and the plane to move
before minimizing the second curvature variation. On \(\Afamily\), the ordinary second
Taylor coefficient is \(a_{\Afamily}=\calQ_{\Afamily}(H,J)/2\).
The correspondence between the calculations and these bounds is as follows.
\begin{center}
\begin{tabular}{@{}>{\raggedright\arraybackslash}p{0.23\linewidth}>{\raggedright\arraybackslash}p{0.41\linewidth}>{\raggedright\arraybackslash}p{0.29\linewidth}@{}}
\toprule
\textbf{Calculation}&\textbf{Required output}&\textbf{Role in the proof}\\
\midrule
Identities on \(\Afamily\):
Lemma~\ref{lem:A-first-null},
\eqref{eq:A-dual-checks}, \eqref{eq:mass-model}
&The first curvature variation vanishes; the dual, projection and
row-parameter identities hold for all parameters.
&Validates the \(\Afamily\) model in Proposition~\ref{prop:A-mass}, used in the lower bound
\eqref{eq:A-witness-inequality}.\\
\addlinespace
Positivity on \(\Afamily\):
\eqref{eq:five-coefficients}, \eqref{eq:I-formula}
&Five angle functions are positive everywhere, and the irrational correction
is bounded.
&Proposition~\ref{prop:A-sign}: \(a_{\Afamily}>37/560\), hence \(\calQ_{\Afamily}(H,J)>1/8\).\\
\addlinespace
Constants on \(\Bfamily\):
the norms following \eqref{eq:B-Cauchy}
&The two explicit norm sums and the final rational inequality are correct.
&The arithmetic in Proposition~\ref{prop:B-bound} gives
\(\calQ_{\Bfamily}(H_N)>1/48\).\\
\bottomrule
\end{tabular}
\end{center}
All three checks use exact rational or integer arithmetic. Exact polynomial sign
verification using Sturm's theorem also appears in \cite[Section 5]{GVZ11}.

The code is included in the ancillary files. The companion repository is
{\urlstyle{tt}\def\UrlBreaks{\do\/\do\-}\url{https://github.com/shengtaoguo/gromoll-meyer-curvature-calculations}}. 
The accompanying documentation relates the code to the numbered formulas
and gives reproduction instructions.

\subsection{Bernstein bounds}

There are two uses of the Bernstein basis. The first removes the row
parameter \(m=|a|^2\in[0,1]\) from the positivity test. Here \((a,b)\)
is the first row of the \(\Sp(2)\) representative, with
\(|a|^2+|b|^2=1\). For the angle pair
\(\theta=(u,v)\), write the polynomial \(P_\theta\) of
\eqref{eq:mass-polynomial} in the form
\[
 P_\theta(m)=B_0(1-m)^2+2B_1m(1-m)+B_2m^2.
\]
The five angle functions in \eqref{eq:five-coefficients}, in their stated
order, are
\[
 (\varphi_0,\ldots,\varphi_4)
 =\left(B_0,\frac{B_0+B_1}{2},\frac{B_0+4B_1+B_2}{6},
              \frac{B_1+B_2}{2},B_2\right).
\]
They satisfy
\[
 P_\theta(m)=\sum_{j=0}^4\binom4j m^j(1-m)^{4-j}\varphi_j(u,v).
\]
For fixed \(u,v\), this is a weighted average of the five values
\(\varphi_j(u,v)\): the weights are nonnegative and sum to one.
Thus their strict positivity proves \(P_\theta(m)>0\) for every
\(m\in[0,1]\), including both endpoints.

The second use proves that each \(\varphi_j\) is positive for every angle
pair. The exact identities under
\(u\mapsto-1/(2u)\) and \(v\mapsto-1/(2v)\) reduce this task to
\(\mathcal B=[-3/4,3/4]^2\), as in Proposition~\ref{prop:A-sign}.
Indeed, if \(|u|>3/4\), its image has absolute value less than \(2/3\),
and infinity maps to zero; the same holds for \(v\).
Clearing the positive denominators and cancelling positive circle factors
gives integer numerators with the same signs as the \(\varphi_j\).
We test these polynomials on rectangles by the following expansion.

For a polynomial \(p\) on a rectangle, an affine change of variables maps
the rectangle to \([0,1]^2\). Write its tensor Bernstein expansion as
\[
 p(x,y)=\sum_{i=0}^d\sum_{j=0}^e
  b_{ij}\binom di x^i(1-x)^{d-i}
             \binom ej y^j(1-y)^{e-j}.
\]
Again the basis functions are nonnegative and their sum is one, so
\(p(x,y)\) is a weighted average of the coefficients \(b_{ij}\).
Positive coefficients therefore imply strict positivity on the closed rectangle. The
power-to-Bernstein conversion in one variable is the elementary identity
\begin{equation}\label{eq:Bernstein-conversion}
 b_j=\sum_{i=0}^j\frac{\binom ji}{\binom di}a_i,
 \qquad p(x)=\sum_{i=0}^d a_i x^i.
\end{equation}
The verifier applies the affine substitution and this formula directly
in each variable, with common positive denominators cleared to integers.

The following preorder list defines a binary subdivision of \(\mathcal B\);
read it from left to right along each row.
At a node \((\ell;\mathcal I)\), first check the Bernstein coefficients
of the functions \(\varphi_j\), \(j\in\mathcal I\), on the current
rectangle. If \(\ell=0\) or \(1\), bisect the \(u\)- or \(v\)-interval,
respectively, and visit the lower half before the upper half;
\(\ell=-1\) marks a terminal rectangle. For example, the first entry
\((1;\varnothing)\) splits the root square at \(v=0\) without testing a
function there. The indices 0 through 4 refer directly to
\(\varphi_0,\ldots,\varphi_4\) above.
\begin{center}
\small
\begin{tabular}{llll}
\toprule
\((1;\varnothing)\)&\((0;\varnothing)\)&\((0;0123)\)&\((1;\varnothing)\)\\
\((-1;4)\)&\((-1;4)\)&\((1;\varnothing)\)&\((-1;4)\)\\
\((-1;4)\)&\((0;0123)\)&\((1;\varnothing)\)&\((-1;4)\)\\
\((-1;4)\)&\((1;\varnothing)\)&\((-1;4)\)&\((-1;4)\)\\
\((0;\varnothing)\)&\((1;23)\)&\((-1;014)\)&\((-1;014)\)\\
\((1;3)\)&\((0;04)\)&\((1;\varnothing)\)&\((-1;12)\)\\
\((-1;12)\)&\((-1;12)\)&\((0;012)\)&\((-1;4)\)\\
\((-1;4)\)&&&\\
\bottomrule
\end{tabular}
\end{center}
Here \(0123\), for example, denotes \(\{0,1,2,3\}\).
On each terminal rectangle, every function has positive Bernstein
coefficients either there or at an ancestor. The twenty-nine nodes give
fifteen terminal rectangles and thirty-eight polynomial bounds.
A bound proved on a parent rectangle holds on all its descendants, so
the five functions on the fifteen terminal rectangles require only these
thirty-eight comparisons. Together, the rectangles cover the entire square;
each comparison proves positivity on a whole rectangle.

\clearpage
\subsection{Verification}

\noindent\textbf{Identities on \(\Afamily\).}
Starting from the tensor entries \eqref{eq:H}--\eqref{eq:J}, the product
metric and frame conversion \eqref{eq:product-metric}--\eqref{eq:frame-conversion},
and the projection, horizontal pair and fibre maps
\eqref{eq:A-push}--\eqref{eq:A-LR}, the program reconstructs the first
curvature variation and the constrained quadratic model for \(h=H/\sqrt2\).
The variables \(u,v\) are the two phase parameters, while
\[
 r=\frac{1-w^2}{1+w^2},\qquad s=\frac{2w}{1+w^2},\qquad m=r^2
\]
parametrize the row amplitudes. The vanishing first variation, all
fifty-four source dual equations, the five homogeneous columns and
projection identities in \eqref{eq:A-dual-checks}, and the decomposition \eqref{eq:mass-model}
are checked over \(\mathbb Q(u,v,w)\). In each case, subtracting the
two sides and clearing denominators gives the zero polynomial.
This checks the whole parameter family. The row poles \(a=0\) and \(b=0\)
are also evaluated directly. Finally, the constant \(16\times16\) block \(B_*\)
used in \eqref{eq:small-witness} is inverted over the rationals, and both
inverse products are verified.

\smallskip
\noindent\textbf{Positivity on \(\Afamily\).}
The witness \eqref{eq:small-witness} is substituted into
\eqref{eq:mass-polynomial} to reconstruct the five functions
\(\varphi_0,\ldots,\varphi_4\). Their primitive integer numerators have
bidegrees \((32,36)\) for indices 0--3 and \((8,8)\) for index 4.
The two antipodal identities are checked as polynomial identities.
There are twenty-five comparisons for indices 0--3 and thirteen for index 4,
so the number of tested coefficients is
\(25(33\cdot37)+13(9\cdot9)=31{,}578\). All are strictly positive. The weighted-average
identity then gives \(P_\theta(m)>0\) for all angles and all
\(m\in[0,1]\).

The program also checks \eqref{eq:I-formula} for \(I=L_{J_i}\) and
the polynomial square identities giving \(|I|\le33/16\).
The normalization \(15/2\) in \eqref{eq:A-witness-inequality} turns the
threshold \(15/16\) in \eqref{eq:mass-polynomial} into a Taylor margin
\((2/15)(15/16)=1/8\). Restoring the coefficient \(\sqrt2\) adds
\((\sqrt2-7/5)I\) to this lower bound. Using \(7/5<\sqrt2<10/7\) gives
\[
 \begin{aligned}
 a_{\Afamily}&>\frac18+\left(\sqrt2-\frac75\right)I
 \ge\frac18-\left(\frac{10}{7}-\frac75\right)\frac{33}{16}
       =\frac{37}{560},\\
 \calQ_{\Afamily}(H,J)&=2a_{\Afamily}>\frac18.
 \end{aligned}
\] 

\smallskip
\noindent\textbf{Constants on \(\Bfamily\).}
Here the program evaluates the two norm sums following
\eqref{eq:B-Cauchy}, obtaining \(173/96+27/16000\) and \(281/96\).
Their weights in \eqref{eq:B-Cauchy} are 2 and \(2/3\), respectively,
so \(C_\eta=50/9+27/8000\). The \(\Bfamily\) source pair has Gram determinant 54,
and the fixed-pair second curvature variation has numerator \(27/4\),
as computed in Section~\ref{sec:B-bound}. Its completed-square estimate
bounds the loss in this numerator from moving the flag by \(C_\eta\).
The final rational calculation is
\[
 \frac18-\frac1{54}\left(\frac{50}{9}+\frac{27}{8000}\right)
 =\frac{85757}{3888000}>\frac1{48}.
\]
Section~\ref{sec:B-bound} proves the source identity
\eqref{eq:B-dual-vector} for every admissible velocity
\(\delta\in\ker\mathsf E\), and hence shows that this number is a
lower bound for \(\calQ_{\Bfamily}(H_N)\).

Section~\ref{sec:completion} combines these two bounds with its uniform
neighborhood estimates. Setting \(t=\varepsilon^3\) then gives
\(\min\sec_{g_\varepsilon}\ge c\varepsilon^6\) for all sufficiently
small \(\varepsilon>0\).

\clearpage
\appendix

\section{The normal quadratic gap}\label{app:clean}

We prove the bracket estimate and the rank calculation used in
Proposition~\ref{prop:clean-zero}.

\subsection{A quantitative comparison for two contractions}

Temporarily allow both contraction factors \(\lambda,\mu\) to lie
in \((0,1)\). Let \(g_2\) be the metric on \(\Sp(2)\) after these
two contractions and before the GM quotient, and define its curvature
numerator by
\[
 \kappa_2(X,Y)=g_2\bigl(R^{g_2}(X,Y)Y,X\bigr).
\]
All bracket norms in this appendix use the original bi-invariant inner
product \(Q\). Put
\(U=X_{\mathfrak p}+X_{\mathfrak q}+\mu X_{\mathfrak h}\),
\(V=Y_{\mathfrak p}+Y_{\mathfrak q}+\mu Y_{\mathfrak h}\), and
\[
 Z=X_{\mathfrak p}+\lambda X_{\mathfrak q}
                    +\lambda\mu X_{\mathfrak h},\qquad
 W=Y_{\mathfrak p}+\lambda Y_{\mathfrak q}
                    +\lambda\mu Y_{\mathfrak h}.
\]
The product curvatures in the two submersions give
\begin{equation}\label{eq:clean-product-bound}
 \kappa_2(X,Y)\ge\tfrac14\norm{E_0}^2
   +\tfrac{\lambda(1-\lambda)^3}{4}\norm{K_0}^2
   +\tfrac{\lambda\mu(1-\mu)^3}{4}\norm{H_0}^2,
\end{equation}
where
\(E_0=[Z,W]\), \(K_0=[U_{\mathfrak k},V_{\mathfrak k}]\), and
\(H_0=[X_{\mathfrak h},Y_{\mathfrak h}]\). Only nonnegative
O'Neill terms have been discarded. For example, the first contraction
has auxiliary factor metric scaled by
\(s=\lambda/(1-\lambda)\); its horizontal lift contributes
\(\lambda^4/(4s^3)=\lambda(1-\lambda)^3/4\) times the bracket
square. The second contraction gives the last coefficient in exactly
the same way, with \(t=\mu/(1-\mu)\) and the already contracted
\(H_\Delta\) metric.

Write
\[
 P_1=[X_{\mathfrak p},Y_{\mathfrak p}],\quad
 C_1=[X_{\mathfrak p},V_{\mathfrak k}]
                    +[U_{\mathfrak k},Y_{\mathfrak p}].
\]
The symmetric-pair relations give
\[
 E_0=P_1+\lambda C_1+\lambda^2K_0,\quad
 C_1=\lambda^{-1}(E_0)_{\mathfrak p},\quad
 P_1=(E_0)_{\mathfrak k}-\lambda^2K_0,
\]
and \([U,V]=P_1+C_1+K_0\). Similarly, if
\(Q_1=[X_{\mathfrak q},Y_{\mathfrak q}]\), then
\(Q_1=(K_0)_{\mathfrak h}-\mu^2H_0\).
Consequently the sum \(\mathscr D\) of the five squared norms in
\eqref{eq:five-defects} obeys
\[
 \mathscr D\le
 \max\{8+3\lambda^{-2},\,6+8\lambda^4,\,1+2\mu^4\}
      (\norm{E_0}^2+\norm{K_0}^2+\norm{H_0}^2).
\]
In our case \(\lambda=\mu=1/2\), this and
\eqref{eq:clean-product-bound} give the explicit bound
\begin{equation}\label{eq:five-square-bound}
 \kappa_2(X,Y)\ge\frac1{2560}\mathscr D.
\end{equation}
The GM quotient adds another nonnegative term. For a \(g_2\)-orthonormal
horizontal frame \(X,Y\), the Gram determinant is one, so
\eqref{eq:five-square-bound} gives the sectional-curvature bound.

\subsection{The linearized defects on \texorpdfstring{\(\Bfamily\)}{B}}

At \eqref{eq:B-zero-form} use arbitrary variations
\[
 \dot U=P(\xi)+Q_-(p)+\mathsf H(q),\qquad
 \dot V=P(\chi)+Q_-(r)+\mathsf H(s).
\]
The last three defects in \eqref{eq:five-defects} impose
\(\chi=\gamma x,p=\alpha y,s=\beta y\), with ranks three, two and
two. The \(\mathfrak k\)-defect then imposes
\(q-r\in\R y\), of rank two. The remaining off-diagonal component
of the full bracket is, up to a fixed common sign,
\[
 y\xi+\xi y+rx+xr+(\beta-\gamma)[x,y].
\]
The first two terms range over \(\spanof_\R\{1,y\}\); since
\(r,x\) are imaginary, \(rx+xr\) is real. The last term supplies
the independent \([x,y]\)-direction. This component has rank three,
giving total rank \(3+2+2+2+3=12\).

For \(\Afamily\), the corresponding last rank-three map becomes
\([x^{-1}\xi,y]+r-\Ad_{x^{-1}}s\) after multiplication by
\(x^{-1}\); its image is all of \(\im\HH\). Also
\(\ip yz=-1/2\) for the unit normalization, so the additional
rank-one condition never degenerates. The kernels of these linear maps
are precisely the tangent variations of the respective bracket-zero
forms. Combining with the six transverse horizontal equations proves
the kernel assertion used in Proposition~\ref{prop:clean-zero}.
The normal derivative has a uniform positive least singular value on
each compact incidence. Taylor expansion of the defects, followed by
\eqref{eq:five-square-bound}, gives the quadratic distance estimate.

\section{Explicit dual vectors and coordinates}\label{app:explicit-dual}

Use the source data of Section~\ref{sec:A-source}, and put
\(n=y\times z\), \(T_\circ=\diag(A_\circ,B_\circ)\),
\(R_A^\circ=\Ad_{A_\circ}\), \(R_B^\circ=\Ad_{B_\circ}\), and
\(R_A=R_A^\circ S_A\). Source components are indexed by
\(G,K_1,K_2,H,L\). Within the G component the rational frame is
\((d_1,d_2,\rho,\omega)\). The fibre multiplier order is
\((q_1,q_2,h_0,\ell,g_0)\).

\subsection{One inhomogeneous source vector}

Split \(\kappa=(\kappa_D,\kappa_U,\kappa_V)\) according to the
fifty-four source variables, and define
\[
 k_U=Q^{-1}\kappa_U,\quad k_V=Q^{-1}\kappa_V,\quad
 \gamma=\Ad_{g_{r,s}}(Q_G^{-1}\kappa_D),
 \quad\widetilde u=\Ad_{T_\circ}(k_{U,G}),\quad
 \widetilde v=\Ad_{T_\circ}(k_{V,G}).
\]
Set
\begin{align*}
 q_\perp&=-\tfrac12a_*\times\gamma_{d_1},\qquad
 A_1=\widetilde v_{d_1}-R_Ak_{V,H}-R_A^\circ k_{V,K_1},\\
 w_0&=\tfrac12\{A_1+(R_A-r^2I)q_\perp\},\qquad
 z_A=a_*\times R_Aa_*,\\
 \tau&=\tfrac23\ip{z_A}{\widetilde u_\omega-rs\gamma_\omega-a_*\times w_0},
 \qquad q=q_\perp+\tau a_*.
\end{align*}
Since \(R_Ay=a_*\), orthogonality gives
\(\ip{a_*}{R_Aa_*}=\ip y{a_*}=1\). Hence \(|z_A|^2=4-1=3\).
Define the V half of the multiplier by
\[
 \lambda^0_V=\bigl(S_A(q-k_{V,H})-k_{V,K_1},\,
 S_B(q-k_{V,H})-k_{V,K_2},\,q-k_{V,H},\,0,\,q\bigr).
\]
Next put
\[
 d=\widetilde v_\omega-rsq,\qquad z_1=-d/2,\quad z_2=d/2,
 \qquad w_1=\tfrac12\{A_1+(R_A-r^2I)q\},
\]
and
\(Z_G^0=\Ad_{T_\circ^{-1}}(z_1,z_2,0,w_1)\).
For the U half set
\begin{align*}
 q_1^U&=(R_A^\circ)^{-1}
          (a_*\times z_1+r^2\gamma_\omega-\widetilde u_{d_1}),\\
 q_2^U&=(R_B^\circ)^{-1}
          (a_*\times z_2+s^2\gamma_\omega-\widetilde u_{d_2}),\\
 b_1&=y\cdot(k_{U,K_1}+q_1^U),\qquad
 b_2=z\cdot(k_{U,K_2}+q_2^U),\\
 h_U&=\tfrac13\{(2b_1+b_2)y+(b_1+2b_2)z\},\qquad
 \lambda_U^0=(q_1^U,q_2^U,h_U,0,\gamma_\omega).
\end{align*}
The other source components are
\begin{align*}
 Z_{K_1}^0&=\tfrac12y\times(k_{U,K_1}+q_1^U-S_Ah_U),\\
 Z_{K_2}^0&=\tfrac12z\times(k_{U,K_2}+q_2^U-S_Bh_U),\\
 Z_H^0&=\tfrac12a_*\times(k_{U,H}+h_U-\gamma_\omega),
 \qquad Z_L^0=0.
\end{align*}
Substitution verifies \eqref{eq:A-dual-checks}. All divisors in these
formulas are constant.

\subsection{Five homogeneous columns}

For \(e,t\in a_*^\perp\) and \(\nu\in\R\), put
\[
 z_0=-\tfrac12a_*\times(e-t),\quad
 q_1=(R_A^\circ)^{-1}e,\quad q_2=(R_B^\circ)^{-1}e,
 \quad h_0=\nu n.
\]
Define a homogeneous multiplier and source vector by
\begin{align*}
 \lambda_U&=(q_1,q_2,h_0,t,t),\qquad\lambda_V=0,\\
 Z_G&=\Ad_{T_\circ^{-1}}D(z_0,z_0),\\
 Z_{K_1}&=\tfrac12y\times(q_1-S_Ah_0),\qquad
 Z_{K_2}=\tfrac12z\times(q_2-S_Bh_0),\\
 Z_H&=\tfrac12a_*\times(h_0-t),\qquad
 Z_L=-\tfrac1{20}a_*\times t.
\end{align*}
The matrices \(K_Z,K_\lambda\) are obtained from the five inputs
\[
 (e,t,\nu)=(n,0,0),\ (y-z,0,0),\ (0,n,0),\ (0,y-z,0),\ (0,0,1).
\]
They are independent of the row amplitudes and satisfy
\(\mathsf B^TQZ+\mathsf E^T\lambda=0\).

\subsection{The constant twenty-dimensional coordinate map}

Let \(d_*=y-z\), and write a vector
in \(\R^{20}\) in the order
\[
 x=(\alpha\in\R^4,\ \beta\in\R^4,e_\beta,t_\beta,
       \gamma\in\R^4,e_\gamma,t_\gamma,\ \nu,\eta_h\in\R^3).
\]
Define the four vector components of the fibre coefficient \(\eta\)
by
\begin{equation}\label{eq:coordinate-table}
 \bigl(R_A^\circ\eta_{q_1},R_B^\circ\eta_{q_2},\eta_\ell,\eta_g\bigr)_j
       =\alpha_j a_*+\beta_j n+\gamma_jd_*,\qquad 1\le j\le4.
\end{equation}
The homogeneous source parameters are
\(e=e_\gamma n-e_\beta d_*/2\),
\(t=t_\gamma n-t_\beta d_*/2\), together with \(\nu\).
Equations~\eqref{eq:coordinate-table} and these two assignments define
\(\mathcal R x\in\R^5\oplus\R^{15}\) in
\eqref{eq:A-model}. The last four coordinates are exactly
\((\nu,\eta_{h,i},\eta_{h,j},\eta_{h,k})\). The polynomial in
\eqref{eq:small-witness} is evaluated in these coordinates.

\section{Intrinsic variations and parameter identities}\label{app:mass-encoding}

Use the \(\Afamily\) parameters in \eqref{eq:A-parameters}, with \(h=H/\sqrt2\)
and the rational splitting \(J=J_r+\sqrt2J_i\).
The first variation of each of \(h,J_r,J_i\) can be computed by
restricting the pulled-back tensor to the complete source flat. For
its projection to \(\Sp(2)\), use
\(x=\bar BA\), \(X=P(x)\), \(D=D(y,z)\), \(H_0=\mathsf H(a_*/2)\), and
\(Y=D+H_0\). That projection is parametrized by
\[
 \psi(\xi,\eta)=g\exp(\xi X)\exp(\eta D)\exp(\eta H_0).
\]
Let \(X_{\xi,\eta},Y_{\xi,\eta}\) be the coordinate tangent fields
of this projection. For \(k\in\{h,J_r,J_i\}\), write the entries
of its pullback as \(k_{11},k_{12},k_{22}\). These are the entries
on the full source flat because the tensor is independent of the L factor.
The full source flat, including its L component, has Gram determinant 15.
Its intrinsic first variation is therefore
\begin{equation}\label{eq:J-intrinsic}
 L_k=\frac1{15}\left(
 \partial_\xi\partial_\eta k_{12}
       -\tfrac12\partial_\xi^2k_{22}
       -\tfrac12\partial_\eta^2k_{11}\right)_{(0,0)}.
\end{equation}
This applies separately to \(J_r,J_i\) after
\eqref{eq:frame-conversion}. In particular, ordinary second Taylor
coefficients of the matrix exponentials already contain their factor
\(1/2\); no further uniform factor of two is inserted in
\eqref{eq:J-intrinsic}.

\subsection{The mixed derivative for H}

Put \(a_*=y+z\), \(d=y-z\), and \(n=y\times z\), in the
normalization \eqref{eq:A-fixed-quaternions}. Write \(x=\rho+\omega\),
with \(\omega\in\im\HH\). The relation \(xy=zx\) gives
\[
 \omega=\alpha a_*+\rho n\quad\text{for some }\alpha\in\R.
\]
The equation \(Bz=a_*B\) gives
\(\ip{\im B}y=0\) and \(\ip{\im B}n=-\re B\).
Set \(y_1=a_*+d/2\), \(z_1=a_*-d/2\), and define the real-linear
functional on quaternions
\begin{align*}
 \ell(c)={}&(\re c)\ip{-4y_1+(107/20)z_1}{\omega}\\
 &+\ip{-(13/15)y_1-4z_1}{(\im c)\times\omega}
                  +4\rho\ip{z_1}{\im c}.
\end{align*}
These are precisely the mixed blocks of \(h\). In the flat coordinates
of Proposition~\ref{prop:A-preservation}, equivariance under
conjugation by \(q_\eta=e^{\eta a_*/2}\) yields
\[
 h_{12}(\xi,\eta)=s_\xi\ell(e^{3\eta a_*/2}B).
\]
Using \(\omega=\alpha a_*+\rho n\) and the two identities for
\(\im B\) in the displayed functional gives
\[
 \ell(a_*B)=\tfrac{35}{3}\rho\re B
                  -\tfrac{35}{6}\alpha\ip{\im B}{a_*}
          =\tfrac{35}{6}\re(Bx).
\]
Since \(s'_0=-r\) and \(Bx=A\), it follows that
\[
 \partial_\xi\partial_\eta h_{12}(0,0)
   =-\tfrac{3r}{2}\ell(a_*B)=-\tfrac{35}{4}\re a.
\]
Also \(h_{11}=0\) and \(h_{22}=(35/2)\re a(\xi,\eta)\),
because \(|z_1|^2=7/2\). The relation
\(\partial_\xi^2a=-a\) proves the other derivative used in
Lemma~\ref{lem:A-first-null}.

\subsection{Dependence on the row parameter}

Put \(\mu=9/25\), and use \(r,s\ge0\) at the three evaluation points
\(m=0,1,\mu\). Define
\begin{align*}
 \mathsf A_0&=\mathsf A(0),\qquad
 \mathsf A_1=\mathsf A(1)-\mathsf A(0),\\
 d_0&=d(0),\qquad
 d_1=\frac{\tfrac54d(\mu)-d_0}{\mu},\\
 b_{\bullet,0}&=b_\bullet(0),\qquad
 b_{\bullet,2}=\frac{b_\bullet(\mu)-(1-\mu)b_\bullet(0)-\mu b_\bullet(1)}
                         {\mu(\mu-1)},\\
 b_{\bullet,1}&=b_\bullet(1)-b_{\bullet,0}-b_{\bullet,2}.
\end{align*}
The values correspond to \(w=1,0,1/2\), respectively, and all
arithmetic remains rational. The factor \(5/4\) divides by
\(s=4/5\) at the middle value. These evaluation formulas define candidate
coefficients; the identity in Proposition~\ref{prop:A-mass} is checked
over the full rational parameter field with an independent symbolic
\(w\). Thus no interpolation assumption is made.

\clearpage
\bibliographystyle{alpha}
\bibliography{references}

\end{document}